\documentclass[10pt, reqno]{amsart}

\usepackage{amsmath,amsthm,amsfonts,bbm, amssymb, amscd, graphicx,color, tikz, mathtools}
\usepackage[colorlinks = true, citecolor = black]{hyperref}
\usetikzlibrary{matrix}

\newcommand{\ii}{\mathbbm{i}}
\newcommand{\q}{\mathbf{q}}
\newcommand{\C}{\mathbb{C}}

\newcommand{\R}{\mathbb{R}}

\newcommand{\N}{\mathbb{N}}
\newcommand{\Sp}{\mathbb{S}}
\newcommand{\T}{\mathcal{T}}
\newcommand{\norm}[1]{\left\vert #1 \right \vert}

\renewcommand{\overline}{\bar}
\renewcommand{\vartheta}{\theta}

\renewcommand{\footnote}[1]{}

\newcommand{\partl}[1]{\frac{\partial}{\partial #1}}

\renewcommand{\ker}{\text{ker }}

\renewcommand{\Re}{\text{Re}}
\renewcommand{\Im}{\text{Im}}

\def\[#1\]{\begin{align}#1\end{align}}
\def\(#1\){\begin{align*}#1\end{align*}}

\newtheorem{thm}{Theorem}
\numberwithin{thm}{section}

\newtheorem{prop}[thm]{Proposition}
\newtheorem{lemma}[thm]{Lemma}
\newtheorem{cor}[thm]{Corollary}

\newtheorem{question}[thm]{Question}

\theoremstyle{definition}
\newtheorem{defi}[thm]{Definition}

\theoremstyle{definition}
\newtheorem{example}[thm]{Example}
\newtheorem{remark}[thm]{Remark}

\numberwithin{equation}{section}
\numberwithin{figure}{section}

\newcommand{\st}{\;:\;}

\begin{document}

\title[QR Mappings on sub-Riemannian Manifolds]{Quasiregular mappings \\on sub-Riemannian manifolds}
\author[K. F\"assler]{Katrin F\"assler}
\author[A. Lukyanenko]{Anton Lukyanenko}
\author[K. Peltonen]{Kirsi Peltonen}

\email{katrin.fassler@helsinki.fi}

\email{anton@lukyanenko.net}

\email{kirsi.peltonen@aalto.fi}

\address{Department of Mathematics and Statistics, University of Helsinki, P.O.B. 68, FI-00014 Helsinki, Finland}
\address{Department of Mathematics, University of Illinois at Urbana-Champaign, 1409 West
Green Street, Urbana, IL 61801, USA}
\address{Department of Mathematics and Systems Analysis, Aalto University, P.O. Box 11100, 00076, Aalto, Finland}


\thanks{K.F. was supported by the Academy of Finland, project number 252293. A.L. was supported by National Science Foundation grants DMS 0838434 and DMS 0901620 and Department of Mathematics and Systems Analysis of Aalto University.
}

\begin{abstract}
We study mappings on sub-Riemannian manifolds which are quasi{-}regular with respect to the Carnot-Carath\'{e}odory distances and discuss several related notions. On H-type Carnot groups, quasiregular mappings have been introduced earlier using an analytic definition, but so far, a good working definition in the same spirit is not available in the setting of general sub-Riemannian manifolds. In the present paper we adopt therefore a metric rather than analytic viewpoint.

As a first main result, we prove that the sub-Riemannian lens space admits nontrivial uniformly quasiregular (UQR) mappings, that is, quasiregular mappings with a uniform bound on the distortion of all the iterates. In doing so, we also obtain new examples of UQR maps on the standard sub-Riemannian spheres. The proof is based on a method for building conformal traps on sub-Riemannian spheres using quasiconformal flows, and an adaptation of this approach to quotients of spheres.

One may then study the quasiregular semigroup generated by a UQR mapping. In the second part of the paper we follow Tukia to prove the existence of a measurable conformal structure which is invariant under such a semigroup. Here, the conformal structure is specified only on the horizontal distribution, and the pullback is defined using the Margulis--Mostow derivative (which generalizes the classical and Pansu derivatives).
\end{abstract}

\date{\today}

\maketitle

\setcounter{tocdepth}{1}

\tableofcontents
\newpage
\section{Introduction}
A homeomorphism $f$ between two metric spaces is \emph{quasiconformal} if it has bounded infinitesimal distortion. That is, there exists a finite constant $H$ such that
\begin{equation} \label{eq:metricQC0}
\limsup_{r \to 0}\frac{\sup_{|x-y|\leq r}|fx-fy|}{\inf_{|x-y|\geq r}|fx-fy|}\leq H\quad\text{for all }x.
\end{equation}
 In many contexts, one broadens the definition to include \emph{locally} quasiconformal mappings, that is, \emph{covering mappings} $f$ satisfying
\begin{equation}  \label{eq:metricQC}
H(x,f):=\limsup_{r \to 0}\frac{\sup_{|x-y|\leq r}|fx-fy|}{\inf_{|x-y|= r}|fx-fy|}\leq H\quad\text{for all }x.
\end{equation}
While $H(x,f)$ is a reasonable quantity to measure the \emph{dilatation} or \emph{distortion} also for a non-injective mapping $f$, the expression appearing in \eqref{eq:metricQC0} would be inappropriate in this regard since $\inf_{|x-y|\geq r}|fx-fy|$ may vanish for arbitrarily small values of $r$ even in points where $f$ is locally injective.

The theory of quasiconformal mappings is well-studied in a variety of contexts, see \cite{MR0385004,KR1,MR979599,KR2,HK,MR1654771,MR1869604}.  The focus of this paper is on \emph{branched covering} mappings satisfying \eqref{eq:metricQC} in the context of sub-Riemannian manifolds.

Recall that a sub-Riemannian manifold is a smooth manifold $M$ with a choice of smooth subbundle $HM \subset TM$ satisfying the H\"{o}rmander condition. That is, the iterated brackets of $HM$ span $TM$. A smooth choice of inner product on $HM$ then gives $M$ a sub-Riemannian path metric (\emph{Carnot-Carath\'{e}odory distance}). We will assume that every sub-Riemannian manifold is equipped with an inner product $g$, corresponding norm $\norm\cdot$, and distance function $d$.

A sub-Riemannian manifold is \emph{equiregular} if the dimension of $H^1M=HM$ and the bundles $H^{i+1}M=\langle H^i M, [HM,H^iM]\rangle$ is constant over all of $M$. One writes $q_i = \dim H^{i}M/H^{i-1}M$, so that the topological dimension of $M$ is given by $\dim M = \sum_i q_i$. The Hausdorff dimension, on the other hand, is given by $Q=\sum_i i q_i$, and the Hausdorff measure with respect to the cc-distance in dimension $Q$ is equivalent to Lebesgue measure on $M$ \cite{aMi85}. Without further specification, the term `measure' on a sub-Riemannian manifold will refer to one of these equivalent measures. We often write $|A|$ for the measure of a set $A$ in a sub-Riemannian manifold.

We say that a mapping between manifolds is a \emph{branched cover} if it is continuous, discrete, open, and sense-preserving (notice that we do not assume a branched cover to be onto). It is known that a branched cover $f$ between manifolds of dimension $n\geq 2$ is a local homeomorphism away from a \emph{branch set} $B_f$ of topological codimension at least two (Chernavskii's theorem, see \cite{Ri}).

\begin{defi}\label{d:metric_qr_subRiem:mfd}
Let $M$ and $N$ be two equiregular sub-Riemannian manifolds. We call a mapping $f: M \rightarrow N$  \emph{$K$-quasiregular} if it is constant or if
\begin{enumerate}
\item $f$ is a branched cover onto its image\\ (i.e., continuous, discrete, open and sense-preserving),
\item $H(\cdot,f)$ is locally bounded on $M$,
\item $H(x,f)\leq K$ for almost every $x\in M$,
\item the branch set $B_f$ and its image have measure zero.
\end{enumerate}
A  mapping is said to be \emph{quasiregular} if it is $K$-quasiregular for some $1\leq K<\infty$.
\end{defi}

\begin{remark}We expect that the last condition in Definition \ref{d:metric_qr_subRiem:mfd} is
unnecessary,
but proving this is beyond the scope of the current paper, as it requires the development of a modulus definition of quasiregularity, see \cite{Ri, MR2585391}. Note also that while in this paper the distortion will always be computed with respect to  Carnot-Carath\'{e}odory distances on the corresponding sub-Riemannian manifolds, the definition could be applied to more general metric spaces as well.
\end{remark}

A non-injective quasiregular mapping $f$ acting on  space $M$ is \emph{uniformly quasiregular} (UQR for short) if, for all $n$, we have that $f^{n}$ is $K$-quasiregular for some $K$ not depending on $n$ (note that a priori it is not clear that the composition of two quasiregular maps is quasiregular). Injective UQR mappings are called \emph{uniformly quasiconformal}.

Denote by $\Lambda$ the \emph{semigroup} generated by a UQR map $f:M\rightarrow M$, that is
\[\Lambda=\{f^n:M\rightarrow M\ |\ f \hbox{ uniformly quasiregular },\ n\in \N \}.\]
The \emph{Fatou set} of $f$ is
then defined as
\[F(f)=\{p\in M\ |\;\mbox{ there is an open }U,\;p\in U,\;\Lambda|_U\mbox{ normal}\},\]
where \emph{normal} means that every sequence of $\Lambda$ contains
a locally uniformly convergent subsequence. The \emph{Julia set}
of $f$ is then defined as $J(f)=M \setminus F(f)$.

 We prove:

\begin{thm}
\label{thm:intro:QRexists}
Every lens space, with its natural sub-Riemannian structure, admits a uniformly quasiregular self-mapping with nonempty branch set and a Cantor set type Julia set.
\end{thm}
We proceed to focus on an invariant of UQR mappings. In the setting of Riemannian spheres, Iwaniec and Martin \cite{MR1404085} proved the existence of an invariant measurable structure for any abelian uniformly quasiregular semigroup. Tukia established earlier the corresponding result for groups of quasiconformal mappings, and \cite{MR1404085} is an adapatation of this approach to the non-injective setting. Beyond the Riemannian setting, the existence of an invariant conformal structure for groups of quasiconformal mappings on compactified Heisenberg groups $\mathbb{H}^n$, $n\in \mathbb{N}$, was first proved by Chow in \cite{MR1329530}. In \cite{MR2927672}, the case of non-injective quasiregular mappings was studied independently for the first Heisenberg group $\mathbb{H}^1$, where invariant conformal structures can be related to invariant CR structures. It is not clear how this relation would generalize to higher dimensions.

In the present paper, generalizing \cite{MR1329530}, we define a \emph{conformal structure} on a sub-Riemannian manifold $M$ as an inner product on the horizontal bundle $HM$, up to rescaling (cf.~ also the `sub-conformal structures' on contact manifolds in \cite[Definition 2.1.~(2)]{MR2293640}).
That is, two inner products $\langle \cdot, \cdot\rangle_1$ and $\langle \cdot, \cdot\rangle_2$ represent the same conformal structure if we have $\langle \cdot, \cdot\rangle_1 = \lambda(x) \langle \cdot, \cdot\rangle_2$ for some positive function $\lambda(x)$ on $M$. We say that a conformal class is \emph{measurable} if it is represented by a measurable choice of an inner product (that is possibly undefined on a set of measure 0).

If the sub-Riemannian manifold $M$ is equipped with an inner product $\langle \cdot, \cdot \rangle$, one can in a natural way interpret a new inner product $\langle \cdot, \cdot\rangle_1$ as function $s:M \to S$ to a higher-rank symmetric space  (see the proof of Theorem \ref{thm:tukia}). The conformal class of $\langle \cdot, \cdot\rangle_1$ is said to be \emph{bounded} (with respect to $\langle \cdot, \cdot\rangle$) if $s$ is a bounded map.

A conformal class represented by $\langle \cdot, \cdot\rangle$ is \emph{invariant} under a quasiregular mapping $f$ if we have
\[
\langle f_*u, f_*v\rangle = \lambda(x) \langle u, v\rangle
\]
for some positive function $\lambda$ and any choice of vectors $u,v\in H_xM$. We formalize the meaning of $f_*v$ in \S \ref{sec:differential} using the Margulis--Mostow derivative (which is equivalent to the Pansu derivative when a Carnot group is concerned). The discussion of differentiability leads to the following question:

\begin{question}
Let $f$ be a mapping between sub-Riemannian manifolds with a continuous Margulis--Mostow derivative. Is $f$ then continuously differentiable in the classical sense? (See \S \ref{sec:bundles}, where we prove the measurability of the Margulis--Mostow derivative.)
\end{question}

After addressing the differentiability questions, we prove in \S \ref{sec:invariant}:
\begin{thm}
\label{thm:intro-conformal}
Every uniformly quasiregular self-mapping $f$ of a sub-Riemannian manifold admits an invariant measurable conformal structure that is bounded with respect to the given inner product.
\end{thm}

The structure of the paper is as follows. After discussing basic properties of QR mappings in \S \ref{sec:QR}, we provide an explicit QR mapping on lens spaces in \S \ref{sec:QR_lens}. We apply a ``conformal trap'' construction in \S \ref{sec:UQR_lens} to transform it into a UQR mapping and thus prove Theorem \ref{thm:intro:QRexists}. Following a review on differentiability in \S \ref{sec:differential}, we prove Theorem \ref{thm:intro-conformal} in \S \ref{sec:invariant}.

\section{Quasiregular mappings}\label{sec:QR}
The theory of quasiregular mappings in $\R^n$ developed as a generalization of the study of complex-analytic mappings of one variable. A crucial result is the following Picard-type theorem (see also \cite{DP}):
\begin{thm}[Rickman]
Let $f: \R^n \rightarrow \R^n$ be a non-constant $K$-quasiregular mapping. Then $f$ misses at most finitely many points, with the number of points bounded above in terms of $K$ and $n$.
\end{thm}

In the classical theory, one starts with the assumption that the quasiregular mapping $f: \R^n \rightarrow \R^n$ is in the Sobolev class $W^{1,n}_{loc}$, with the differential $df$ satisfying the condition $\norm{df}^n \leq K Jf$ almost everywhere. One then proves that the analytic assumptions are equivalent to the topological and metric properties that we take in this paper as the definition of quasiregularity.

Recall that Carnot groups (with the Heisenberg group as the primary example) provide a generalization of Euclidean space admitting a sub-Riemannian metric. Namely, these are simply connected nilpotent Lie groups that admit a left-invariant sub-Riemannian metric and a one-parameter family of homotheties.

Heinonen and Holopainen first generalized the notion of quasiregularity and proved basic properties of quasiregular mappings on Carnot groups in \cite{MR1630785}. Their assumptions were later relaxed by Dairbekov in \cite{MR1721676}. Quasiregular mappings satisfying these weaker conditions were then studied in general Carnot groups and the restricted setting of the Heisenberg group and other two-step Carnot groups  in \cite{MR1847514,MR2252688,MR2316342} and other works. As in the classical case, these authors started with the analytic definition of quasiregularity. This definition does not easily generalize to arbitrary sub-Riemannian manifolds.

Following the Heinonen--Koskela approach (see \cite{HK}) of rephrasing quasiconformality in terms of strictly metric properties, we start with a \textbf{metric} definition of quasiregularity. Note that our definition is akin to the characterization of Euclidean quasiregular mappings in \cite{MR0259114}, see also \cite{Ri}, and to the definition of quasiregularity in \cite{MR2585391}.

So far, important analytic properties -- such as Sobolev regularity and the $K_O$-inequality -- of metrically defined quasiregular mappings are only known
in Carnot groups, see the discussion in \cite[4]{MR2679532}, or for Euclidean source spaces, which is the case considered in
\cite{MR2585391}. We conjecture that analogous properties hold for mappings between equiregular sub-Riemannian manifolds, but do not pursue this further in the present paper.

\subsection{Rectifiable curves}
\label{sec:rectifiable}
We first recall some terminology and notation for curves, spheres, and balls in metric spaces (see also for instance \cite{MR1421821}).

Let $M$ be a sub-Riemannian manifold and $\gamma:[a,b]\rightarrow M$ a continuous curve. Recall that the \emph{length} of $\gamma$ is defined as
\begin{equation*}
\mathrm{length}(\gamma) = \sup \sum d(\gamma(t_i),\gamma(t_{i-1}))\leq \infty,
\end{equation*}
where the supremum is taken over all partitions of $[a,b]$. A curve $\gamma:[a,b] \to M$ is a \emph{geodesic} if $d(\gamma(a),\gamma(b))=\mathrm{length}(\gamma)$. By the Hopf-Rinow theorem, any pair of sufficiently nearby points in $M$ can be joined by a geodesic. If $M$ is complete, then in fact \textbf{any} two points can be joined by a geodesic.

Recall further that $\gamma$ is said to be \emph{rectifiable} if $\mathrm{length}(\gamma)<\infty$, in which case it can be given a Lipschitz reparametrization. An absolutely continuous curve $\gamma:[a,b] \to M$ in a sub-Riemannian manifold, in particular a Lipschitzly parametrized curve, is rectifiable if and only if it is horizontal, that is, $\dot{\gamma}(t)\in H_{\gamma(t)}M$ for a.e.\ $t$. In this case the length of $\gamma$ can be computed using the sub-Riemannian inner product $g$ as
\begin{equation*}
\mathrm{length}(\gamma) = \int_a^b \norm{\dot{\gamma}(t)} dt.
\end{equation*}

We will denote balls and spheres (respectively) in a metric space $(M, d)$ as follows:
\begin{equation*}
B(x,r)=\{y\in M:\; d(x,y)<r\}\quad\text{and}\quad S(x,r)=\{y\in M:\; d(x,y)=r\}.
\end{equation*}

\subsection{Definitions of dilatation}
In this subsection, we clarify the different definitions of dilatation in the introduction and show that a quasiregular map is locally quasiconformal away from the branch set. We then provide a way to compute the maximum dilatation of a quasiregular map under some smoothness conditions.

\begin{lemma}\label{l:ll}
Let $Y$ be a locally geodesic space and $f:X \to Y$ a homeomorphism. Then for every $x\in X$ and $r>0$ small enough (depending continuously on $x$),
\begin{equation}
\label{eq:ellellstar}
\ell := \inf_{|x-y|\geq r}|fx-fy|\geq \inf_{|x-y|=r}|fx-fy|=\ell^{\ast}.
\end{equation}
\begin{proof}\footnote{Jeremy: notation is confusing. Too many $x$, $x'$, $y$, etc. K: not sure what is best here. (i) write $X'$ for the target space and use $x,y,z$ for the three relevant points in the source, or (ii) leave target space $Y$ and consider $x,x',x''$ in $X$ (in any case I think the minimizing point can be denoted by the same letter as what will appear in \eqref{eq:ellellstar})}
Without loss of generality, we may assume $Y$ that is geodesic. To see this, let $U$ be a geodesic open neighborhood of $fx$. Since both $\ell$ and $\ell^\ast$ approach $0$ as $r \rightarrow 0$, we have that for sufficiently small $r$, the infima in \eqref{eq:ellellstar}  remain unchanged if solely points $y$ in $f^{-1}(U)$ are considered.\footnote{K: changed this sentence}
We may therefore replace $X$ with $f^{-1}U$ and $Y$ with $U$.

Fix $0<r \leq \text{diam}(X)$, and let $x'\in X$ minimize the expression $\inf_{|x-y|\geq r}|fx-fy|$ (perhaps up to a small error term). Let $\gamma$ be a 
geodesic\footnote{K: replaced `curve' by `geodesic' here}
 joining $fx'$ to $fx$, and $f^{-1}\gamma$ its preimage in $X$. By continuity of the metric on $X$, there is an $x''$ on $f^{-1}\gamma$ with $\norm{x-x''}=r$. Since $\gamma$ is a geodesic, we have $\norm{fx-fx''}\leq \norm{fx-fx'}$.
\end{proof}
\end{lemma}

\begin{prop}\label{p:qr_loc_qc} Let $f:M\to N$  be a quasiregular map between equiregular sub-Riemannian manifolds. Then almost every point in $M$ possesses a neighborhood $U$ such that $f: U \rightarrow f(U)$ is quasiconformal.
\end{prop}
\begin{proof}
The statement follows from the definition of $B_f$ as the set of points where $f$ does not define a local homeomorphism. Since the branch set is closed, each point $x\in M\setminus B_f$ possesses a neighborhood $U_x$ restricted to which $f$ is a homeomorphism and, in fact quasiconformal by the local and essential boundedness of $H(\cdot,f)$. The only point worth noting here is that the definition of quasiconformality on metric spaces requires the boundedness of the distortion function \eqref{eq:metricQC0}, whereas quasiregularity yields only that the a priori smaller function $H(\cdot,f)$ as defined in \eqref{eq:metricQC} is bounded. Lemma \ref{l:ll} ensures that we obtain the same class of quasiconformal mappings for both distortion functions.

Now if we consider an open subset $\Omega$ in a sub-Riemannian manifold $M$, there are two natural distance functions on $\Omega$: the restriction of the cc-distance $d_M$ and the cc-distance $d_{\Omega}$ which is given by the restriction of the sub-Riemannian metric $g_M$ to $\Omega$. Since $M$ is locally geodesic, the two metrics agree on small balls inside $\Omega$. By Lemma \ref{l:ll} the uniform boundedness of $H(\cdot,f)$, where the distortion is computed with respect to $d_M$ and $d_N$, then implies the uniform boundedness of the `quasiconformal distortion' \eqref{eq:metricQC0} in terms of $d_{\Omega}$ and $d_{f(\Omega)}$.
\end{proof}

In the study of uniformly quasiregular mappings later on, the precise value of a mapping's distortion will be relevant. For this purpose we record the following result.

\begin{prop}\label{p:smooth_K_qr}
Let $f:M\to N$ be a quasiregular map between sub-Riemannian manifolds such that each point in $M\setminus B_f$ possesses an open neighbourhood restricted to which $f$ is a diffeomorphism and  in which the functions
\(
&\lambda_{-}(y) = \inf_{v\in H_yM} \frac{ \norm{f_{\ast}v}_N}{ \norm{v}_M}
&\lambda_{+}(y) = \sup_{v\in H_yM} \frac{ \norm{f_{\ast}v}_N}{ \norm{v}_M}
\)
(with $v$ always nonzero, and $f_{\ast}$ denoting the usual derivative)
are continuous and satisfy
\begin{equation*}
\frac{\lambda_{+}(y)}{\lambda_{-}(y)}\leq K.
\end{equation*}
Then  $f$ is $K$-quasiregular.
\end{prop}

\begin{proof}
Since the branch set of a quasiregular map is of zero measure, it suffices to show that
\begin{equation*}
H(x,f)\leq K,\quad \text{for all }x\in M\setminus B_f.
\end{equation*}
Let $x$ be a point in $M\setminus B_f$. Then there exists an $r>0$ such that $f|_{B(x,r)}$ is a diffeomorphism onto its image and $\lambda_-$, $\lambda_+$ are continuous functions on  $\overline{B}(x,r)$ whose quotient satisfies the desired bound. We may further assume that $r$ has been chosen small enough so that $x$ can be joined to any point $y$ in $B(x,r)$ by a geodesic $\gamma:[0,\ell] \to M$.
Then
\begin{align*}
d_N(f(x),f(y)) &\leq  \int \norm{f_{\ast} \dot{\gamma}(t)}_Ndt\\
                    &\leq  \int \lambda_{+}(\gamma(t))\norm{\dot{\gamma}(t)}_M dt\\
                    &\leq  \left(\sup_{t\in [0,\ell]} \lambda_{+}(\gamma(t))\right)\left( \int \norm{\dot{\gamma}(t)}_M dt\right)\\
                    &= \sup_{B(x,r)} \lambda_{+} \cdot d_M(x,y).
\end{align*}

For the reverse estimate, we notice that there is $r'<r$ such that for $z\in \overline{B}(x,r')$, the points $f(x)$ and $f(z)$ can be joined by a geodesic $\gamma':[0,\ell']\to N$ which lies entirely inside $f(B(x,r))$. Since $f$ is a local horizontal diffeomorphism, $\gamma'$ is the image of the horizontal curve $f^{-1}\circ \gamma'$ and
therefore
\begin{align*}
d_N(f(x),f(z))&= \mathrm{length}(\gamma')\\
&\geq \inf_{t\in [0,\ell']} \lambda_{-} (f^{-1}(\gamma'(t)) )\cdot\mathrm{length} (f^{-1} \circ \gamma')\\
&\geq \inf_{B(x,r)} \lambda_{-}  \cdot d_M(x,y)	
\end{align*}
It follows
\begin{equation*}
\frac{d_N(f(x),f(y))}{d_N(f(x),f(z)))}\leq \frac{\sup_{B(x,r)}\lambda_+}{\inf_{B(x,r)}\lambda_-}
\end{equation*}
and thus by letting $r\to 0$,
\begin{equation*}
H(x,f) \leq\frac{\lambda_+(x)}{\lambda_-(x)},
\end{equation*}
as desired.
\end{proof}

\subsection{Lusin's conditions}

A quasiconformal map between equiregular sub-Rie{\-}mannian manifolds is absolutely continuous in measure, see Theorem 7.1 in \cite{MargulisMostow95}. Since the inverse of such a map is again quasiconformal by Corollary 6.4 in \cite{MargulisMostow95}, we have both Lusin's condition  $(N)$ and $(N^{-1})$. That is, the image, respectively preimage, under a quasiconformal map of each set of measure zero is a set of measure zero. The same holds true for non-constant quasiregular mappings.

\begin{lemma}\label{l:lusin}
Let $f:M\to M$ be a non-constant quasiregular mapping on an equiregular sub-Riemannian manifold $M$ and let $E\subset M$ be a set with $|E|=0$. Then $|f(E)|=|f^{-1}(E)|=0$.
\end{lemma}

\begin{proof}
By Proposition \ref{p:qr_loc_qc}, each point in $M\setminus B_f$ possesses a  neighbourhood restricted to which $f$ is quasiconformal. Since $M$ is strongly Lindel\"of, that is, every open cover of an arbitrary open subset of $M$ has a countable subcover, we may cover the open set $M\setminus B_f$ by a countable union of sets $\{B_i\}_{i\in\mathbb{N}}$ so that $f|_{B_i}:B_i\to f(B_i)$ is quasiconformal. By the absolute continuity of quasiconformal maps, we have
\begin{equation*}
|f(E\cap B_i)|= |(f|_{B_i})^{-1}(E)|=0,\quad\text{for all }i\in \mathbb{N}
\end{equation*}
and thus, by the assumptions on the branch set of a quasiregular map and its image,
\begin{equation*}
|f(E)|\leq |f(B_f)|+\sum_{i=1}^{\infty} |f(E\cap B_i)|
\end{equation*}
and
\begin{equation*}
|f^{-1}(E)|\leq |B_f|+ \sum_{i\in\mathbb{N}}|f^{-1}(E)\cap B_i|=|B_f|+ \sum_{i\in\mathbb{N}}|(f|_{B_i})^{-1}(E\cap f(B_i))|,
\end{equation*}
which concludes the proof.
\end{proof}

\subsection{Radially bilipschitz mappings}
\label{sec:QSP}

Although very natural,  Definition \ref{d:metric_qr_subRiem:mfd} is not easy to work with: it can be difficult to verify directly that a given map is quasiregular, and it is not clear in general whether the composition of two quasiregular maps is again quasiregular. For that reason we introduce a sub-class of quasiregular mappings that are easier to handle: the \emph{radially bilipschitz mappings} with $|B_f|=|f(B_f)|=0$. Examples of such mappings are those of bounded length distortion with finite multiplicity, as well as open, discrete and sense-preserving locally $L$-bilipschitz maps. Both the quasiregular and uniformly quasiregular mappings we construct below will be radially bilipschitz.

\begin{defi}
Let $M$ and $N$ be two sub-Riemannian manifolds endowed with Carnot-Carath\'{e}odory distances. We say that a branched cover $f:M\to N$ is \emph{radially bilipschitz} (RBL for short) if there exist constants $0<c_1<c_2<\infty$ and for each $x\in M$ a number $r(x)>0$ such that for all $y\in M$ with $d(x,y)\leq r(x)$, one has
\begin{equation}\label{eq:rbl}
c_1 d(x,y)\leq d(f(x),f(y))\leq  c_2 d(x,y).
\end{equation}
\end{defi}

We emphasize that the condition \eqref{eq:rbl} is formulated with respect to the base point $x$, in particular we do not require that the restriction of $f$ to any curve emanating from $x$, or even to a small neighborhood around $x$, would be bilipschitz.

\begin{example}
The prototypical example of RBL maps is the planar map $$f(z)=z^2/\norm{z}.$$
In polar coordinates, the map is given by $(r,\theta)\mapsto(r, 2\theta)$, so that away from the origin the map is locally 2-bilipschitz. At the origin, one sees branching but no dilatation: a circle of any radius is mapped onto itself.
\end{example}

Immediately from the definition of RBL maps, we obtain the following criterion for quasiregularity.

\begin{prop}\label{p:BLD_QR}
Any RBL map $f:M \to N$ with $|B_f|=0$ and $|f(B_f)|=0$ between sub-Riemannian manifolds is quasiregular.
\end{prop}

We state several useful properties of RBL maps.

\begin{prop}\label{p:qsp_comp}
Let $f:M \to M'$ and $g:M' \to M''$ be RBL maps between sub-Riemannian manifolds. Then $g\circ f: M \to M''$ is RBL.
\end{prop}

\begin{prop}\label{p:restr}
If $f:M\to N$ is RBL and $\Omega \subset M$ is a domain, then $f|_{\Omega}:\Omega \to N$ is RBL (with respect to the restriction of either the distance function or sub-Riemannian inner product on $M$).
\end{prop}

\begin{prop}\label{p:qsp_glue}
Let $M$ and $N$ be sub-Riemannian manifolds of equal dimension. Assume that $M$ can be written as a finite union of open sets, $M=\bigcup_{i=1}^k B_i$ and suppose further that we are given RBL maps $f_i:B_i \to N$ so that ${f_i}|_{B_i \cap B_j} = {f_j}|_{B_i \cap B_j}$. Then the map $f:M\to N$, defined by $f|_{B_i}=f_i$ is RBL.
\end{prop}

\begin{proof} The condition on $B_i \cap B_j$ ensures that $f$ is well-defined and continuous. By \cite[Remark 3.2]{MR1909607}, every discrete and open map between oriented generalized $n$-manifolds is sense-preserving or sense-reversing (so that if each $f_i$ is sense-preserving, so is $f$). Thus, regarding the topological properties, it is sufficient to verify openness and discreteness. For that purpose, let $U$ be an open set in $M$, then
\begin{equation*}
f(U)=f\left(\bigcup_{i=1}^k U \cap B_i\right)=\bigcup_{i=1}^k f_i(U\cap B_i)
\end{equation*}
is open as a union of open sets. The discreteness is also immediate since for all $y\in N$, we have
\begin{equation*}
f^{-1}(\{y\})=\bigcup_{i=1}^k f_i^{-1}(\{y\}),
\end{equation*}
which is discrete as a finite union of discrete sets. The metric condition in the definition of RBL is easy to verify, as it is a local condition and the sets $B_i$, on which $f_i$ is RBL, are open.
\end{proof}

\subsection{Examples of RBL mappings}

An important class of RBL maps are those of \emph{bounded length distortion} introduced in Definition \ref{d:BLD} below, which is analogous to \cite[Definition 0.1]{MR1909607}.

\begin{defi}\label{d:BLD}
Let $M$ and $N$ be two sub-Riemannian manifolds endowed with Carnot-Carath\'{e}odory distances. We say that a mapping $f:M\to N$ is of bounded length distortion (BLD) if it is continuous, open, discrete, sense-preserving and there exists $L\geq 1$ such that
\begin{equation*}
\frac{1}{L} \mathrm{length}(\gamma) \leq
 \mathrm{length}(f\circ \gamma)
 \leq L  \mathrm{length}(\gamma)
\end{equation*}
for all continuous paths $\gamma$ in $M$.
\end{defi}

\begin{prop}\label{tP:bld_qr}
Let $M$ and $N$ be oriented sub-Riemannian $n$-manifolds and assume that $M$ is complete.  Let $f:M\to N$ be a surjective $L$-BLD map of finite multiplicity. Then $f$ is RBL.
\end{prop}

\begin{proof}
The $L$-Lipschitz continuity is immediate: Since $M$ is a geodesic space, for every $x,y\in M$, there exists $\gamma$ such that
\begin{equation*}
d_M(x,y)=\mathrm{length}(\gamma).
\end{equation*}
Then
\begin{equation*}
d_N(f(x),f(y))\leq \mathrm{length} (f \circ \gamma) \leq L \mathrm{length}( \gamma) = L d_M(x,y).
\end{equation*}

We now prove the local \textbf{lower} bound for $d_N(f(x),f(y))$ at an arbitrarily chosen point $x\in M$. Since $f$ has finite multiplicity, $f^{-1}\{f(x)\}$ consists of finitely many points $x,x_1,\ldots, x_N$. By \cite[Proposition 4.13]{MR1909607}, we have
\[\label{eq:incl_balls}
B(x,\tfrac{\rho}{C})\cup \bigcup_{i=1}^N B(x_i,\tfrac{\rho}{C})
&\subseteq f^{-1}(B(f(x),\rho))\\
&\subseteq
B(x,C\rho) \cup \bigcup_ {i=1}^N B(x_i,C\rho)\nonumber
\]
for all $\rho>0$ and some constant $C\geq 1$ which depends on the data associated to $M,N,f$ (including multiplicity).
There exists $\rho_1>0$ such that the balls $B(x,2C\rho)$ and $B(x_i,2C\rho)$ are pairwise disjoint for $\rho<\rho_1$. Set
\begin{equation*}
r:= 2 C \rho_1.
\end{equation*}
Now let $y$ be an arbitrary point in $M$ with $r':=d_M(x,y)<r$. If $x=y$, there is nothing to prove. Otherwise, we have
$y\in M \setminus  \left(B(x,C\rho)\cup \bigcup B(x_i,C\rho)\right)$ and $f(y)\in N \setminus B(f(x),\rho)$ for $\rho=r'/(2C)$ by \eqref{eq:incl_balls} and so
\begin{equation*}
d_N(f(x),f(y))> \rho = \frac{r'}{2C}=\frac{d_M(x,y)}{2C},
\end{equation*}
which concludes the proof.
\end{proof}

From Proposition \ref{p:BLD_QR} and Proposition \ref{tP:bld_qr}, we obtain the following result.

\begin{cor}\label{c:BLD_QR}
Let $M$ and $N$ be oriented sub-Riemannian $n$-manifolds  and assume that $M$ is complete. Let $f:M\to N$ be a  surjective BLD map of finite multiplicity with $|B_f|=|f(B_f)|=0$.
Then $f$ is quasiregular.
\end{cor}

A second class of RBL mappings is provided by locally bilipschitz maps.

\begin{lemma}\label{l:horiz_lip}
Let $M$ and $N$ be compact sub-Riemannian manifolds and let $f:M\to N$ be a smooth map which is \emph{horizontal}, that is, $f_{\ast,p}(H_p M) \subseteq H_{f(p)}N$ for all $p\in M$, then $f$ is Lipschitz.
\end{lemma}

\begin{proof}
Since $f$ is smooth, we can consider the pull-back metric $f^{\ast} g_N$. This is a semidefinite form (on the subbundle $HM$). As $M$ is compact, the same holds true for the unit horizontal subbundle and the map $(p,v)\mapsto \frac{(f^{\ast}g_N)_p(v,v)}{g_{M,p}(v,v)}$ defined thereon is continuous and finite. Thus there exists a constant $0\leq L <\infty$ such that
\begin{equation}
 \norm{f_{\ast} v}_N\leq L \norm{v}_M\quad\text{for all }v\in HM.
\end{equation}
The claim follows since $f$ is horizontal.
\end{proof}

\begin{cor}\label{c:contacto_bld}
Let $M$ be a compact sub-Riemannian manifold whose sub-Rie{\-}mannian structure is given by a contact form. Then every sense-preserving contactomorphism $f:M \to M$ is BLD, RBL and quasiconformal.
\end{cor}

\begin{proof}
We apply Lemma \ref{l:horiz_lip} to both $f$ and its inverse in order to see that $f$ is bilipschitz. It follows that $f$ is  distorting the length of curves only by a controlled amount. Thus, $f$ is BLD, RBL and quasiconformal.
\end{proof}

\subsection{Remarks on QR mappings}

There is a well-established \textbf{analytic}  definition available for quasiregularity on the \textbf{Heisenberg group} and a rather complete theory has been developed for mappings which are quasiregular according to this definition \cite{MR1778673}. A continuous mapping between equiregular sub-Riemannian manifolds which are modelled conformally and locally bilipschitzly on $\mathbb{H}^n$ (`Heisenberg manifolds') could now be called $K$-quasiregular if it is $K$-quasiregular in charts according to the analytic definition. Similarly, one could study quasiregular mappings on manifolds which are modelled on other two-step Carnot groups, based on the notions and results in \cite{MR1847514}.

\begin{question}
Does this definition of `quasiregularity in charts' on Heisenberg manifolds yield the same class of mappings as Definition \ref{d:metric_qr_subRiem:mfd}?
\end{question}


Mappings on the Heisenberg group which are quasiregular according to the analytic definition stated in \cite{MR1778673} behave analogously as their Euclidean counterparts: non-constant quasiregular mappings satisfy Lusin's condition and their branch set is of vanishing measure, it is also known that the composition of two such mappings is again quasiregular. If we start with a metric distortion condition on a sub-Riemannian manifold, our knowledge is less advanced.

\begin{question}
\label{q:QRanalytic}
What are geometric or analytic characterizations of mappings which are quasiregular in the sense of Definition \ref{d:metric_qr_subRiem:mfd}?
\end{question}

An answer to Question \ref{q:QRanalytic} would most likely yield that quasiregular mappings are closed under compositions, and provide us with other useful properties.

As we show in \S \ref{sec:QR_lens}, nontrivial quasiregular mappings  exist between any two lens spaces (see Theorem \ref{thm:multitwistQR}). On the other hand, for certain Carnot groups any quasiregular mapping is in fact conformal \cite{MR1630785}. More generally, one can ask:

\begin{question}
For which sub-Riemannian manifolds $M, N$ does there exist a non-constant quasiregular mapping $f: M \rightarrow N$? Does a cohomological obstruction exist  for  non-constant QR mappings $f:G\rightarrow N$, where $G$ is a Carnot group, as for (Riemannian) QR-elliptic manifolds in \cite{MR1846030}?
\end{question}

If a compact Riemannian manifold $M^n$ supports a nontrivial UQR map, then $M$ is elliptic \cite{IM2} (that is, admits a quasiregular mapping from $\R^n$). A detailed proof based on applying Zalcman's rescaling principle on the Julia set can be found in \cite{Ka}. The corresponding statement in the sub-Riemannian setting is an open question.
\begin{question}
Suppose $M$ is an equiregular sub-Riemannian manifold that admits a nontrivial UQR map and that has at almost every point a fixed Carnot group $G$ as a tangent space in the sense of Margulis and Mostow. Does $M$ then admit a quasiregular mapping from $G$?
\end{question}

\section{QR mappings on spheres and lens spaces}\label{sec:QR_lens}

In this section we introduce the setting of sub-Riemannian spheres and lens spaces and provide nontrivial examples for quasiregular mappings on such manifolds. For Riemannian lens spaces, the existence of nontrivial UQR maps follows from \cite{MR1718708} and various properties of such mappings have been studied. For instance, G.~Martin and the third author have shown that the $2$-torus cannot appear as a Julia set either on the standard $\Sp^3$ or on a three-dimensional lens space equipped with the quotient metric \cite{Martin-Peltonen}.
In the present paper, we initiate the study of quasiregular mappings on \textbf{sub}-Riemannian lens spaces. The UQR maps constructed here provide also new Riemannian counterparts. All these mappings have Cantor set type Julia sets.

Recall that a differentiable map $f: M \rightarrow N$ between sub-Riemannian manifolds of the same dimension is locally isometric (in the metric sense) if one has $f_*(HM)\subset HN$ and also $f^* g_N = g_M$. More generally, the mapping $f$ is said to be \emph{conformal} if one has $f_*(HM)\subset HN$  and $f^* g_N = \lambda g_M$, that is, there exists a positive function $\lambda:M\to \R$ such that
 for all $p\in M$ and $v,w\in H_p M$ one has $g_{N}(f_{\ast}v,f_{\ast}w)= \lambda(p) g_{M}(v,w)$.

Notice\footnote{Jeremy: Already true for conformal maps in $\R^n$ for $n\geq 3$. K: not sure I understand this comment. The idea was just to state the implication ``conformal diffeo $\Rightarrow$ $1$-QC'', which holds on all (sub)-Riemannian manifolds, in particular on $\mathbb{R}^n$, $n\geq 1$, no talk about Liouville-type rigidity, but we could of course mention this} that a conformal diffeomorphism is $1$-quasiconformal with respect to the Carnot-Carath\'{e}odory distances defined by the corresponding sub-Riemannian metrics. The assumed smoothness might seem a too strong regularity condition. It is, however, natural, for instance in the context of rigidity results for conformal mappings on Carnot groups \cite{MR2272973} and for isometries on equiregular sub-Riemannian manifolds \cite{uCLD13}. Later in the paper, we consider mappings for which the conformality condition above is formulated in terms of the so-called Margulis--Mostow differential and required to hold only almost everywhere.

\subsection{Sub-Riemannian spheres and lens spaces}
The \emph{sphere} $\Sp^{2n+1}\subset \C^{n+1}$ has a natural sub-Riemannian structure. Namely, one obtains the horizontal subbundle by taking a maximal complex subspace of $T\Sp^{2n+1}$: $H\Sp^{2n+1} = T\Sp^{2n+1}\cap \ii T\Sp^{2n+1}$. The Euclidean inner product on $\C^{n+1}$ then restricts to $H\Sp^{2n+1}$ as a sub-Riemannian metric $g_E$ and corresponding norm $\norm{\cdot}_E$.

We will next take a quotient of $\mathbb{S}^{2n+1} \subset \C^{n+1}$. Let $p>1$ be an integer and $q_1, \ldots, q_{n+1}\in \N$ relatively prime to $p$. Set $\q = (q_1, \ldots, q_{n+1})$, and define
\[R_{p, \q}(z_1,z_2,\ldots,z_{n+1}) = (e^{2\pi \ii q_1 / p}z_1, e^{2\pi \ii q_2/ p}z_2, \ldots, e^{2\pi \ii q_{n+1}/p}z_{n+1}).
\]

Note that the transformation $R_{p, \q}$ preserves the unit sphere $\Sp^{2n+1}$, and that its restriction to $\Sp^{2n+1}$ is conformal with respect to the standard sub-Riemannian structure. Furthermore, it has finite order $p$ and no fixed points on $\Sp^{2n+1}$. The associated quotient space $L_{p, \q}:=\Sp^{2n+1}/\langle R_{p, \q}\rangle$ is called a \emph{lens space}.
Lens spaces are well-studied, especially in dimension 3, see \cite{MR1867354}.

We can give $L_{p,\q}$ a natural sub-Riemannian structure using the following construction.

\begin{prop}\label{p:quotient}
Let $M$ be an equiregular (simply connected) sub-Riemannian manifold and $\Gamma$ a discrete group of isometries of $M$, acting freely and properly discontinuously on $M$.

Let $\pi:M\to M/\Gamma$ be the usual
quotient
map. Then the following holds:
\begin{enumerate}
\item $M/\Gamma$ can be given the structure of a sub-Riemannian manifold which can be endowed with a natural metric so that $\pi$ becomes a locally isometric covering.
\item Every $K$-quasiregular map $F: M \rightarrow M$ with
\begin{equation}\label{eq:cond_well_def}
\pi(p)=\pi(q)\quad \Rightarrow \quad F(p)=F(q),
\end{equation}
induces a $K$-quasiregular map $f:M/\Gamma \to M$, defined by $f([p]):= F(p)$.
\end{enumerate}
\end{prop}

\begin{proof}
Under the given assumptions it is possible to endow $M/\Gamma$ with a sub-Riemannian metric $g$ by setting
\begin{equation*}
g_{[p]}(v_{\ast},v'_{\ast}):= g_p(v,v'),\quad \text{for }p\in {M},\;v,v'\in H_p{M}
\end{equation*}
 with $\pi(p)=[p]$, $\pi_{\ast}v=v_{\ast}$, $\pi_{\ast}v'=v'_{\ast}$,
where $\pi:M\to {M}/\Gamma $ denotes the standard projection given by the quotient map $p\mapsto [p]$, which in this case becomes a locally isometric covering.

Let now $F: M \rightarrow M$ be a $K$-quasiregular map satisfying condition \eqref{eq:cond_well_def}. Then $f:M/\Gamma \to M$ is well-defined by setting $f([p]):= F(p)$. For $r>0$ small enough and $d_{M/\Gamma}([p],[q])\leq r$, we have
\begin{equation*}
d_{M/\Gamma}([p],[q])= d_M(p,q)
\end{equation*}
for $p$ and $q$ in the same fundamental region. Then the $K$-quasiregularity of $f$ follows from the corresponding property of $F$.
\end{proof}

\begin{example}\label{ex:rot}
Let $p>1$ be an integer, and $\q=(q_1, \ldots, q_{n+1})$ with each $q_i$ relatively prime to $p$. Let $R_{p,\q}$  be the associated rotation of $\Sp^{2n+1}$ and $L_{p, \q}$ be the associated lens space. For any $\theta_1, \ldots, \theta_{n+1}\in \R$, the rotation
\[\label{eq:rotation}
R_{\theta_1, \ldots, \theta_{n+1}}(z_1, \ldots, z_{n+1}) = (e^{\ii \theta_1}z_1, \ldots, e^{\ii \theta_{n+1}}z_{n+1})\]
of $\C^{n+1}$ induces a rotation of $\Sp^{2n+1}$ that commutes with $R_{p,\q}$ and is conformal with respect to the standard metric on $\Sp^{2n+1}$ (it is, in fact, an isometry). We thus have an induced ``rotation'' isometry on the lens space $L_{p,\q}$.

The sphere $\Sp^{2n+1}$ admits a larger family of conformal ``rotations''. Namely, the group $U(n+1)$ of unitary matrices acts transitively on $\Sp^{2n+1}$ by isometries. However, this action does not descend to $L_{p,\q}$ because it does not commute with $R_{p,\q}$.
\end{example}

\begin{example}
\label{ex:loxodromic}
The sphere $\Sp^{2n+1}\subset \C^{n+1}$ admits a one-parameter family of ``loxodromic'' conformal maps defined, for $-\infty< d< \infty$, by
\(
T_d(z_1, \ldots, z_{n+1}) = \frac{\left(\cosh(d) z_1+ \sinh(d), z_2, z_3, \ldots, z_{n+1}\right)}{\sinh(d)z_1+\cosh(d)}.
\)
For $d>0$, the transformation $T_d$ maps the ``right hemisphere'' $\Re (z_1)>0$ to a smaller neighborhood of the point $(1,0,\ldots,0)$. Because $T_d$ is complex-analytic (as a map on $\C^{n+1}$), it preserves the horizontal distribution of $\Sp^{2n+1}$. Because it is linear-fractional, $T_d$ is, in fact, conformal with respect to the sub-Riemannian metric. See \cite{KR1} for the case $n=1$.

We remark that \textbf{inside} the unit ball, the transformation $T_d$ is an isometry of complex hyperbolic space with translation length $d$ (see \cite[\S 6.2]{MR1695450} or \cite{MR2987619}). Furthermore, under stereographic projection, $T_d$ corresponds to a dilation map $\delta_r$ of the Heisenberg group.
\end{example}

\subsection{Multi-twist mappings}
We define a multi-twist map of $\mathbb{S}^{2n+1}$ as follows:

\begin{defi}
For $a\in \mathbb{Z}$, the \emph{multi-twist  map} $F_a:\mathbb{S}^{2n+1} \to \mathbb{S}^{2n+1}$ is given by
\begin{equation}\label{eq:multi_twist}
F_a(r_1 e^{\ii\theta_1},\ldots,r_{n+1} e^{\ii\theta_{n+1}})=(r_1 e^{a \ii \theta_1},\ldots, r_{n+1} e^{\ii a \theta_{n+1}}).
\end{equation}
\end{defi}

The goal of this section is to prove (see Corollary \ref{c:twist_qr} and Lemma \ref{l:lens}):
\begin{thm}
\label{thm:multitwistQR}
The mapping $F_a: \Sp^{2n+1}\rightarrow \Sp^{2n+1}$ is quasiregular for each $a$. For each $p$ dividing $a$ and any choice of vector $\q$ of integers relatively prime to $p$, the map $F_a$ induces quasiregular mappings $f_a: L_{p,\q} \rightarrow \Sp^{2n+1}$ and $\pi \circ f_a: L_{p,\q}\rightarrow L_{p,\q}$.
\end{thm}

To describe the properties of $F_a$, it is convenient to work on the domain
\begin{equation*}
U:=\{z=(z_1,\ldots,z_{n+1})\in\mathbb{C}^{n+1}:\; |z_i|\neq 0\quad\text{for }i=1,\ldots,n+1\},
\end{equation*}
in polar coordinates $r_1,\ldots,r_{n+1},\theta_1,\ldots,\theta_{n+1}$. Clearly, $F_a$ is everywhere differentiable on $U$ and for all $z\in U$,
\begin{equation}\label{eq:contact_twist}
(F_a)_{\ast,z}(H_z \mathbb{S}^{2n+1})\subseteq H_{F_a(z)}\mathbb{S}^{2n+1}.
\end{equation}
This can be seen most easily from the following characterization of the horizontal distribution:

\begin{lemma}The horizontal bundle $H\mathbb{S}^{2n+1}$ is given by
\begin{equation*}
H\mathbb{S}^{2n+1}= \mathrm{ker}\alpha\quad\text{with }\alpha= \frac{\ii}{2}\left(\sum_{i=1}^{n+1} z_i d\bar z_i - \bar z_i d z_i\right)=\sum_{i=1}^{n+1}r_i^2 d\theta_i.
\end{equation*}
\begin{proof}
One can write $H\mathbb{S}^{2n+1}= \ker dr \circ J$, where $r(z_1,\ldots,z_{n+1})= (\sum_{i=1}^{n+1} z_i \bar z_i)^{1/2}$ and the mapping $J$ is the restriction to the sphere of the multiplication by $\ii$.  Since $d(r^2)=2r dr$ and $r\equiv 1$ on $\mathbb{S}^{2n+1}$, we find $dr = \frac{1}{2}\sum_{i=1}^{n+1} z_i d\bar z_i+\bar z_i dz_i$, and the expression of $\alpha$ in Cartesian coordinates follows. The expression in polar coordinates is obtained by noting that $z_i= r_i e^{\ii\theta_i}$.
\end{proof}
\end{lemma}

We observe $(F_a)^{\ast}\alpha = a \alpha$. Moreover,
a direct computation yields
\begin{equation*}
g_E((F_a)_{\ast} v, (F_a)_{\ast} v)= \sum_{i=1}^{n+1} \left((\cos \theta_i v_{x_i} + \sin \theta_i v_{y_i})^2+a^2(\sin \theta_i v_{x_i} - \cos \theta_i v_{y_i})^2\right),
\end{equation*}
where
\begin{equation*}
v= \sum_{i=1}^{n+1} v_{x_i} \partial_{x_i}+ v_{y_i}\partial_{y_i}.
\end{equation*}
Hence
\begin{equation}\label{eq:metric_twist}
\norm{v}_E \leq \norm{(F_a)_{\ast} v}_E \leq a \norm{v},\quad v\in H_p \mathbb{S}^{2n+1},\;p \in U.
\end{equation}

If $F_a$ were smooth, \eqref{eq:metric_twist} would give immediately that $F_a$ is $a$-Lipschitz and a BLD map, but we cannot assume that the curves we are considering lie entirely in $U$. We will use the following result, which is immediate for curves contained in the domain where $F_a$ is smooth. The situation is more subtle when a curve meets the set
\begin{equation*}
B:= \mathbb{S}^{2n+1}\setminus U = \{z\in \mathbb{S}^{2n+1}:\; z_i = 0\text{ for some }i\in \{1,\ldots,n+1\}\},
\end{equation*}
but even if $F_a$ is not differentiable in such points, it is ``differentiable along the considered curves''.

\begin{prop}\label{p:twist_horiz_curve_pres}
A curve $\gamma:[a,b] \to \mathbb{S}^{n+1}$ is horizontal if and only if $\gamma'=F_a \circ \gamma$
is horizontal. In this case,
\begin{equation}\label{eq:metric_ineq_curves_twist}
\norm{\dot{\gamma}(s)}_E \leq \norm{\dot{\gamma'}(s)}_E \leq a \norm{\dot{\gamma}(s)}_E
\end{equation}
for almost every $s\in [a,b]$.
\end{prop}

\begin{proof}
Notice that $F_a(B)=B$ and $F_a(\mathbb{S}^{2n+1}\setminus B)=\mathbb{S}^{2n+1}\setminus B$. Each point in $\mathbb{S}^{2n+1}\setminus B$ possesses a neighborhood restricted to which $F_a$ is a diffeomorphism. This shows immediately that if $\gamma$ is a curve contained entirely in $\mathbb{S}^{2n+1}\setminus B$, then it is horizontal if and only if $\gamma'$ is. So let us concentrate on the case where the curve intersects $B$. We notice
\begin{equation*}
\gamma(s)\in B\quad \Leftrightarrow \quad \gamma'(s)\in B.
\end{equation*}
Even more, using standard coordinates from the ambient space $\mathbb{C}^{n+1}$, the $i$-th component of $\gamma$ vanishes at some point $s$ if and only if the $i$-th component of $\gamma'$ vanishes. Denote now by $\gamma_i$ the $i$th component of $\gamma$, and by $\gamma'_i$ the $i$th component of $\gamma'$. The strategy is the following: Assume first that $\gamma$ is horizontal. Then the set of points in $[a,b]$ where all the components of $\gamma$ are differentiable is a full measure set $S\subset [a,b]$. Let us look at $s\in S$. If $\gamma_i(s)\neq 0$, then also $\gamma_i'$ is nonvanishing and differentiable at $s$ with
\begin{equation}\label{eq:metric_ineq}
\dot{\gamma}_i(s)^2 \leq \dot{\gamma}_i'(s)^2 \leq a^2 \dot{\gamma}_i(s)^2
\end{equation}
and
\begin{equation}\label{eq:horizontal_tangent}
a\left({\gamma}_i(s) \dot{\bar \gamma}_i(s)- \bar{\gamma}_i(s)\dot{\gamma}_i(s)\right)=
{\gamma'}_i(s) \dot{\overline{\gamma'}}_i(s)- \overline{\gamma'}_i(s)\dot{\gamma'}_i(s).
\end{equation}
Let us now consider $s$ such that $\gamma_i(s)=0$. We can neglect the case where $s$ is an isolated zero, because there can be at most countably many such in $[a,b]$, and hence they form a set of measure zero. So assume that $s$ is not an isolated zero. Since $\gamma_i$ is by assumption differentiable at $s$ and it vanishes at $s$ as well as on a sequence of points converging to $s$, we must have $\dot{\gamma}_i(s)=0$. In order to prove that  also $\gamma'_i$ is differentiable at $s$ with vanishing tangent (so that \eqref{eq:metric_ineq} and \eqref{eq:horizontal_tangent} remain valid in this case), we resort to the following auxiliary result.
\begin{lemma}\label{l:diff_at_zero}
Let $z:[-\varepsilon,\varepsilon] \to \mathbb{C}$ be differentiable at $0$ with $\dot{z}(0)=0$, and $\eta: \R \rightarrow \R$ a smooth function. Then the function $w:[-\varepsilon,\varepsilon] \to \mathbb{C}$, defined (using any branch of the logarithm) by
\begin{equation*}
w(s):= \left\{
\begin{array}{ll}
r(s)e^{\ii \eta(\varphi(s))}&\text{if }z(s)=r(s)e^{\ii \varphi(s)}\\
0&\text{if }z(s)=0 \end{array}
\right.
\end{equation*}
is differentiable in the point $0$ as well, and $\dot{w}(0)=0$.
\end{lemma}
\begin{proof}
This follows immediately from the definition of derivative.
\end{proof}

Applying this lemma to the non-isolated zeros of $z$ concludes the proof of \eqref{eq:metric_ineq} and \eqref{eq:horizontal_tangent} for almost every $s\in [a,b]$. Thus, if $\gamma$ is horizontal, so is $\gamma'$ by \eqref{eq:horizontal_tangent}, and \eqref{eq:metric_ineq}
gives \eqref{eq:metric_ineq_curves_twist}.

Second, one has to show that $\gamma$ is horizontal whenever the image curve $\gamma'=F_a \circ \gamma$ is. Lemma \ref{l:diff_at_zero} has been formulated in a general setting, so that it can also be applied here, and the rest of the proof goes completely analogously, so we do not carry it out here.
\end{proof}

\begin{cor}\label{c:twist_curves}
$F_a$ is $a$-Lipschitz with
\begin{equation*}
\mathrm{length}(\gamma) \leq \mathrm{length}(F_a \circ \gamma) \leq a \cdot \mathrm{length}(\gamma).
\end{equation*}
for every path $\gamma:[a,b] \to \mathbb{S}^{2n+1}$.
\end{cor}

\begin{proof}
This is an immediate consequence of Proposition \ref{p:twist_horiz_curve_pres} and the discussion of rectifiable curves in \S \ref{sec:rectifiable}.
\end{proof}

\begin{cor}\label{c:twist_qr}
The multi-twist map is a BLD (and thus RBL) map on $\mathbb{S}^{2n+1}$ and $|a|$-quasiregular.
\end{cor}

\begin{proof}
Apply Proposition \ref{p:smooth_K_qr} and Corollary \ref{c:twist_curves}.
\end{proof}

If $a\in p\mathbb{Z}$ for some positive integer $p$, then $F_a$ induces a well-defined
map on the lens space, namely
\begin{equation}\label{eq:multi_twist_lens_sphere}
f_a:L_{p,\mathbf{q}}\to \mathbb{S}^{2n+1},\quad f_a([z]):= F_a(z)\quad\text{for }z\in\mathbb{S}^{2n+1}.
\end{equation}
 Denoting by $\pi:\mathbb{S}^{2n+1}\to L_{p,\mathbf{q}}$ the usual projection, we obtain a multi-twist map of the lens space as
\begin{equation}\label{eq:multi_twist_lens}
\pi\circ f_a: L_{p,\mathbf{q}}\to L_{p,\mathbf{q}}.
\end{equation}

\begin{lemma}\label{l:lens}
Let $a\in p\mathbb{Z}$. The multi-twist maps $f_a:L_{p,\q}\to \mathbb{S}^{2n+1}$ and $\pi\circ f_a:L_{p,\q}\to L_{p,\q}$ are RBL and quasiregular.
\end{lemma}

\begin{proof}
The first statement follows from Corollary \ref{c:twist_qr} and Proposition \ref{p:quotient}. The quasiregularity of $\pi \circ f_a$ is an easy consequence of the fact that $f_a$ is quasiregular and the projection $\pi$ is isometric. Indeed, the Carnot-Carath\'{e}odory distance of $f_a(z)$ and $f_a(w)$ on $\mathbb{S}^{2n+1}$ equals the distance of $\pi(f_a(z))$ and $\pi(f_a(w))$ on $L_{p,\bf{q}}$ if only $f_a(z)$ and $f_a(w)$ are close enough, which can be arranged by continuity. Since $\pi$ and the branches of its inverse are locally isometric,
they are RBL, and so the RBL property of $f_a$ and $\pi\circ f_a$ follows from the corresponding property of $F_a$.
\end{proof}

\section{UQR mappings on spheres and lens spaces}\label{sec:UQR_lens}

There are essentially two methods known to produce UQR mappings: \emph{Latt\`{e}s construction} and the \emph{conformal trap method}. The proof of Theorem \ref{thm:intro:QRexists} uses the conformal trap construction, which
was first introduced by Iwaniec and Martin in \cite{MR1404085} and \cite{MR1454921} to prove the following theorem.

\begin{thm}[Iwaniec, Martin]
\label{thm:classical-trap}
Let $n\geq 2$ and let $f: \Sp^n \rightarrow \Sp^n$ be a quasiregular map of the Riemannian sphere. Then there exists a UQR mapping $g: \Sp^n \rightarrow \Sp^n$ with the same branch set as $f$.
\end{thm}

We start with a hypothetical example illustrating the conformal trap method.
\begin{example}
\label{ex:conformalTrap}
Suppose we have a quasiregular planar map $f: \R^2 \rightarrow \R^2$ with the following properties:
\begin{enumerate}
\item The point $0$ has two preimages $f^{-1}(0) = \{z_1, z_2\}$.
\item The closures of the unit balls $B_0:=B(0,1)$, $B_1:= B(z_1, 1)$, $B_2:=B(z_2,1)$, and $B_3:=B(f(0), 1)$ are disjoint and do not intersect the branch set.
\item The map $f$ respects the balls: $f(B_1)=f(B_2)=B_0$, $f^{-1}(B_0) = B_1 \cup B_2$, and $f(B_0)=B_3$.
\item \label{condition4}The restriction $f\vert_{\partial B_i}$ is a Euclidean translation for $i=0, 1, 2$.
\end{enumerate}
Under these (idealized) conditions, we can apply the conformal trap method. First, we define a new map $g_1$ that agrees with $f$ outside of the balls $B_i$, and inside the balls $B_i$ is given by a Euclidean translation provided by Condition \eqref{condition4}. Next, we write $g=\iota \circ g_1$, where $\iota(z) = \frac{1}{z}$ is inversion in the unit circle.

It is now straightforward to check that $g$ is UQR. Namely, a point inside $B_0$ is trapped in $B_0$ and sees no distortion. A point inside $B_1$ or $B_2$ likewise sees no distortion, while a point outside of $B_0\cup B_1 \cup B_2$ is sent into $B_0$ and does not see any distortion under further iterations. Furthermore, because we are working in the Euclidean plane, the map $g$ obtained in this way is quasiregular provided that $f$ was quasiregular.
One also notes that the branch sets of $f$ and $g$ agree.
\end{example}

In practice, most assumptions of Example \ref{ex:conformalTrap} are violated for an arbitrary QR mapping, making UQR mappings hard to find. For example, a QR map will in general not send balls to balls and will not necessarily have nice behavior on the boundary of a ball. In past works, these issues were resolved using Sullivan's Annulus Extension Theorem \cite{MR658932}, or in the Heisenberg group setting \cite{MR2927672} by a local quasiconformal flow method. Furthermore, a conformal inversion map $\iota$ is not available on most metric spaces, so that the method is most clearly applicable for mappings between spheres and, as the third author first demonstrated in  \cite{MR1718708}, their quotients.

On sub-Riemannian manifolds, already the existence of a non-injective quasiregular map is not immediate. Indeed, the only prior examples are provided by \cite{MR1630785}, which discusses QR mappings in Carnot groups, and \cite{MR2927672}, which discusses UQR mappings on the compactified Heisenberg group.

Here, we will be working with the multi-twist map on sub-Riemannian spheres \eqref{eq:multi_twist}  and lens spaces \eqref{eq:multi_twist_lens_sphere}.  The inversion $\iota$ will be obtained from the antipodal map on $\Sp^{2n+1}$, which under stereographic projection  (see \cite[p.328,329]{KR1}) becomes the \emph{Kor\'{a}nyi inversion} of the Heisenberg group, see \cite[p. 35]{KR2}.

No analogue of Sullivan's Annulus Extension Theorem is known to hold on sub-Riemannian manifolds, so we are not able to conclude that \textbf{every} non-injective quasiregular map on the sub-Riemannian sphere has a UQR counterpart. We can, however, prove that such a counterpart exists  for mappings with a particular symmetry like the multi-twist map which are defined on a contact manifold, where it is possible to construct the quasiconformal extension using local contact flows.

\subsection{Contact flows}\label{ss:flow}

Recall that a smooth compactly supported vector field $W$ in $\mathbb{R}^n$ generates a  global flow $g_s$ with $\partl{s} g_s = W$ for all $s \in \R$, where $g_s: \mathbb{R}^n \to \mathbb{R}^n$ is a diffeomorphism. A corresponding statement holds for smooth vector fields on compact manifolds.

Not every vector field on a sub-Riemannian manifold generates a flow of quasiconformal mappings (with respect to the sub-Riemannian metric). The primary obstruction comes from the fact that a quasiconformal maps must preserve the horizontal bundle. In particular, if $M$ is a contact manifold, any quasiconformal diffeomorphism of $M$ must be a contact transformation.

Libermann \cite{MR0119153} has provided a characterization of vector fields on a  smooth compact contact manifold $(M,\alpha)$ which generate a flow of contact transformations.
We follow the presentation in \cite{MR1340848}. Let $T$ be the Reeb vector field, uniquely determined by the conditions $\alpha(T)=1$ and $d\alpha(T,\cdot)=0$. We consider further
\[
\iota: \mathrm{ker}\alpha \to \{\omega\in \wedge^1 M: \omega(T)=0\},\quad W\mapsto d\alpha(W,\cdot)
\]
and let $\sharp$ be its inverse.

\begin{thm}[Libermann]
\label{thm:Libermann}
Let $(M,\alpha)$ be a smooth compact contact manifold. Then,
\begin{enumerate}
\item any smooth vector field $W$ generating a flow of contact transformations of $M$ is of the form
\[\label{eq:contact_vfd}
    W=\rho T +  \sharp((T\rho)\alpha -d\rho),\]
where $\rho = \alpha(W)$,
\item any smooth vector field of the form \eqref{eq:contact_vfd} for a smooth function $\rho:M\to \mathbb{R}$ generates a flow of contact transformations of $M$ and $\alpha(W)=\rho$.
\end{enumerate}
The function $\rho$ is called \emph{potential} of the vector field $W$.
\end{thm}

In \cite{MR1340848}, the condition has been phrased explicitly for the \textbf{sub-Riemannian sphere} $\mathbb{S}^{3}$:
\begin{prop}
A smooth vector field $W$ on $\Sp^3$ generates a smooth $1$-parameter group of contact transformations if and only if
\begin{equation}\label{eq:vect_sphere}
W=\ii(\bar Z\rho)Z-\ii(Z\rho)\bar Z+\rho T,
\end{equation}
for a smooth function $\rho: \Sp^3 \to \mathbb{R}$, where
\begin{equation*}
Z=\bar z_2 \frac{\partial}{\partial z_1}-\bar z_1 \frac{\partial}{\partial z_2},\quad\text{and}\quad T=-2\Im(z_1\frac{\partial}{\partial z_1}+ z_2 \frac{\partial}{\partial z_2}).
\end{equation*}
\end{prop}

\begin{remark}

Based on Libermann's work, Kor\'anyi and Reimann have established the theory of quasiconformal flows on the \textbf{Heisenberg group}, which is a (noncompact) contact manifold. Their work provides in particular a mild condition which ensures uniqueness and quasiconformality of the flow and also a way to estimate the distortion of the quasiconformal maps produced by such flows, based on the second horizontal derivatives of the potential $\rho$, see Theorem H in \cite{KR2}.
\end{remark}

\subsection{Local flows}\label{s:flow}

We will apply the flow technique to locally modify certain quasi{\-}regular mappings at a point where they are locally diffeomorphic. In this situation, we will have to work with local flows since we cannot expect the flow curves to exist for all times.
For a point $p$ in the domain $\Omega$ of the vector field $W$ under consideration, let $I_p$ be the maximal interval around the origin so that the associated flow curve $s\mapsto g_s(p)$ exists for all $s\in I_p$, in other words, the Cauchy problem
\begin{equation*}
\left\{\begin{array}{l}\frac{\partial}{\partial s}g_s(p)=W_{g_s(p)}\\g_0(p)=p\end{array}\right.
\end{equation*}
has a (unique)
 solution on $I_p$. For $s\in\mathbb{R}$, we denote the set of points $p$ for which $g_s(p)$ is defined by
 \begin{equation*}
 \Omega_s:=\{p\in \Omega: s\in I_p\}.
 \end{equation*}
 Notice that $\Omega_0=\Omega$ and $\Omega_{s}\supseteq \Omega_{s'}$ for $0\leq s<s'$.

\begin{defi}
Let $\Omega \subset \mathbb{S}^{2n+1}$ be a nonempty domain  and let $f: \Omega \to \mathbb{S}^{2n+1}$ be quasiconformal onto its image. We say that $f$ \emph{embeds in a quasiconformal flow} if there is a vector field $W$ on $\Omega$ such that the associated flow $g_s$ satisfies:
\begin{enumerate}
\item $g_s:\Omega_s\rightarrow \mathbb{S}^{2n+1}$ is quasiconformal onto its image for all $s\in [0,1]$ with $\Omega_s\neq \emptyset$,
\item $g_0= \text{identity}$,
\item $g_{1}|_{\Omega_1}=f|_{\Omega_1}$.
\end{enumerate}
Suppose further that for each $s\in [0,1]$ we have that $g_s(p_*)=p_*$ for some $p_*\in \Omega$. We then say that \emph{$f$ embeds in a quasiconformal flow fixing $p_*$}.
\end{defi}

In certain situations, the lack of an annulus extension for quasiconformal mappings on the sub-Riemannian sphere can be overcome by the following construction.

\begin{lemma}
\label{lemma:deformation1}
Let $\Omega\subset \mathbb{S}^{2n+1}$ be open, $p_{\ast} \in \Omega$, and $f: \Omega\rightarrow \mathbb{S}^{2n+1}$ quasiconformal onto its image with $f(p_{\ast})=p_{\ast}$.
Suppose further that $f$ embeds in a quasiconformal flow which fixes $p_{\ast}$, and one is given a ball $B$ centred at $p_{\ast}$ so that the closure of $B$ is completely contained in $\Omega_1\subset \Omega$.
Then there exists a quasiconformal RBL map $g_1: \Omega \rightarrow \mathbb{S}^{2n+1}$ and a ball $B'\subset B$ centred at $p_{\ast}$ so that:
\begin{enumerate}
\item $g_1\vert_{\Omega\backslash B} = f\vert_{\Omega\backslash B}$,
\item $g_1\vert_{B'}= \mathrm{id}$.
\end{enumerate}
\end{lemma}

\begin{proof}
Since $f$ can be locally obtained by a flow, there is a corresponding contact vector field locally generating $g_s$ with $g_1=f$. The identity, on the other hand, is trivially obtained from the constant zero vector field. Since we are interested in a \textbf{quasiconformal} extension, it is not sufficient to consider an arbitrary smooth interpolation between the two vector fields. Instead, we will reconcile the corresponding potential functions by means of a bump function and work with the contact vector field associated to this modified potential.

By assumption $\overline{B}\subset \Omega_1$ and so $g_s(\partial B)$ exists for all $s\in [0,1]$ and is contained in $\Omega$. The set
\begin{equation*}
K:=\cup \{g_s(\partial B):\; s\in [0,1]\}
\end{equation*}
is compact
 and disjoint from $p_{\ast}$, we can thus find a small ball $B'$ centred at $p_{\ast}$ so that $K\cap B'=\emptyset$. Let then $\varphi$ be a bump function associated to $\Omega \setminus B'$ and $K$, that is, a smooth function $\varphi: \mathbb{S}^{2n+1} \to [0,1]$ such that
 \begin{enumerate}
 \item $\mathrm{spt}\varphi \subseteq \Omega \setminus \overline{B'}$
 \item $\varphi|_{K}\equiv 1$.
 \end{enumerate}
Consider the contact vector field $W$ on $\Omega$ and an associated potential $\rho=\alpha(W)$ given by the quasiconformal flow in which $f$ embeds. We define the modified potential $\rho':\mathbb{S}^{2n+1}\to \mathbb{R}$ by setting
\begin{equation*}
\rho':=\left\{\begin{array}{ll}0&\text{on  }B'\cup (\mathbb{S}^{2n+1}\setminus \Omega)\\ \varphi \rho&\text{on }\Omega \setminus (B' \cup K)\\ \rho&\text{on }K\end{array}\right.
\end{equation*}
The so defined $\rho'$ is a smooth, compactly supported function and accordingly, the corresponding vector field $W'$ will generate a flow $(g_s)_s$ of quasiconformal transformations. Since
\begin{equation*}
W'|_{B'}\equiv 0\quad \text{and}\quad W'|_{K}\equiv W,
\end{equation*}
the uniqueness of the flow
and the fact that $B'$ and $K$ have been chosen such that $g_s(B')=B'$ and $g_s(\partial B)\subset K$ for all $s\in [0,1]$, imply
\begin{equation*}
g_1|_{B'}=\mathrm{id}|_{B'}\quad \text{and}\quad g_1|_{\partial B}=f|_{\partial B}.
\end{equation*}
As $g_1$ has the correct boundary values on $B$, we may glue it to the map $f$ in the following way, where, by abuse of notation, we write again $g_1:\Omega \to \mathbb{S}^{2n+1}$ for the modified map,
\begin{equation*}
g_1:=\left\{\begin{array}{ll}g_1&\text{on }{B}\\f&\text{on }\Omega\setminus B\end{array}\right.
\end{equation*}

The RBL property follows from the respective properties of the identity, of $f$ and of the contact transformations by Proposition \ref{p:qsp_glue}.
\end{proof}

We point out that the conditions which we have imposed on the map $f$ in Lemma \ref{lemma:deformation1} are rather restrictive. Not only do we require that a contact vector field $W$ can be found so that $f$ would embed in a flow of $W$, we also ask that a specific point is fixed by all flow maps. Although the latter in not the case for the multi-twist map \eqref{eq:multi_twist} on the sphere, we can still make use of this construction by composing with appropriate conformal maps so as to reduce to a situation where one point is fixed for all times. This works thanks to the special symmetry of the multi-twist map.

\subsection{Auxiliary results}\label{ss:conf_trap}

Before heading to the proof of Theorem \ref{thm:intro:QRexists}, we record a few auxiliary results. The first one shows how to apply Lemma \ref{lemma:deformation1} in a specific situation to find a quasiconformal interpolation between the multi-twist map and a conformal map, which can be different from the identity. Note that the lemma is stated in broad generality to allow us to work with lens spaces in the proof of Theorem \ref{thm:intro:QRexists}.

\begin{lemma}\label{l:ext_multi_twist}

Let $F_a:\mathbb{S}^{2n+1}\to \mathbb{S}^{2n+1}$ be the multi-twist map defined in \eqref{eq:multi_twist}.  Furthermore, fix:
\begin{enumerate}
\item A point $z_*=(r_{1*} e^{\theta_{1*}\ii}, \ldots,  r_{n+1*} e^{\theta_{n+1*}\ii})\in \Sp^{2n+1}\backslash B_{F_a}$,
\item An open set $U\subset \Sp^{2n+1}$ containing $z_*$.
\item A ball $B=B(z_*,R) \subset U$.
\end{enumerate}
Assume furthermore that $\overline B(F_a(z_*), R)\subset F_a(U)$ and $R<R_0$.

Then there exists a quasiregular RBL map $G_1: U \rightarrow \Sp^{2n+1}$ and a ball $B' \subset B$ such that:
\begin{enumerate}
\item $G_1(z_*)=F_a(z_*)$,
\item $G_1|_{B'}$ is an isometry (in particular, conformal),
\item $G_1|_{U\setminus B}=F_a|_{U\setminus B}$.
\end{enumerate}
\end{lemma}
Here, $R_0$ and the bound for the radius of $B'$ depend only on the values of $r_{1*}, \cdots, r_{n+1*}$.
\begin{proof}
We would like to apply Lemma \ref{lemma:deformation1}. Recall that the map $F_a$ is given by:
\(
F_a(r_1 e^{\ii\theta_1},\ldots,r_{n+1}e^{\ii\theta_{n+1}})=(r_1 e^{\ii a\theta_1},\ldots,r_{n+1}e^{\ii a\theta_{n+1}}).
\)
We need to show that $F_a$ locally embeds in a flow. Consider the set
\begin{equation*}
\Omega:=\{(r_1e^{\ii \theta_1},\ldots,r_{n+1}e^{\ii \theta_{n+1}}):\; \theta_i \in (-\pi,\pi), r_i \neq 0, r_1^2+\cdots r_{n+1}^2=1\}.
\end{equation*}
On $\Omega$, we can define a family of functions $H_s$ interpolating between $F_a$ and the identity map:
\begin{equation*}
H_s(r_1e^{\ii \theta_1},\ldots,r_ne^{\ii \theta_n})=(r_1 e^{\ii a^s\theta_1},\ldots,r_{n+1}e^{\ii a^s\theta_{n+1}})\text{, for }s\in[0,1].
\end{equation*}
It is easy to see that the family $H_s$ corresponds locally to a flow along a vector field. We would like to apply Lemma \ref{lemma:deformation1} on a neighborhood of $z_*$ using the family $H_s$, but this would require that $H_s(z_*)=z_*$ for each $s$, which we cannot assume.

Consider instead the point
$z':=(r_{1*}, \cdots, r_{n+1*})$. Clearly, $H_s(z')=z'$ for all $s$. The rotational isometry $R_{\theta_{1*}, \ldots, \theta_{n+1*}}$  of $\Sp^{2n+1}$ (see \eqref{eq:rotation})  relates $z'$ and $z_*$:
\(
R_{\theta_{1*}, \ldots, \theta_{n+1*}}(z')=z_*.
\)
Indeed, we have the following stronger property for the multi-twist map:
\begin{equation}
\label{eq:invariance}
F_a = R_{a\theta_{1*},\ldots,a\theta_{n+1*}} \circ F_a \circ R_{-\theta_{1*}, \ldots,-\theta_{n+1*}}.
\end{equation}
This allows us to appy Lemma \ref{lemma:deformation1} in a neighborhood of $z'$ rather than $z_*$, and then conjugate by the rotation maps to obtain the desired map $G_1$.

Indeed, note that the family $H_s$ satisfies the assumptions of Lemma \ref{lemma:deformation1} with respect to the fixed point $z'$. Hence, we can find small enough balls and a quasiconformal RBL map $H_1$ on $\Omega_1$ so that
\begin{equation}
\label{eq:modification}
H_1:=
\left\{\begin{array}{ll}
F_a&\text{outside }B(z',R)\\
\mathrm{id}&\text{on }B(z',r)
\end{array}\right..
\end{equation}
Here, $0<r<R\leq R_0$, where $R_0$ is small enough so that $B(z',R_0)\subset \Omega_1$, that is, $H_s(z)$ exists for all $z\in B(z',R_0)$ and all $s\in [0,1]$.

We can now extend $H_1$ by $F_a$ in $\mathbb{S}^{2n+1}\setminus \Omega_1$ and define
\(
G_1 :=  R_{a\theta_{1*},\ldots,a\theta_{n+1*}} \circ H_1 \circ R_{-\theta_{1*}, \ldots,-\theta_{n+1*}}.
\)
By \eqref{eq:invariance} and \eqref{eq:modification}, this is the desired modification of $F_a$. Namely, $G_1$ satisfies:
\begin{equation*}
G_1= \left\{\begin{array}{ll}
R_{\theta_{1*}(a-1), \ldots,\theta_{n+1*}(a-1)}&\text{on }B(z_*,r)\\
\text{QC interpolation}&\text{on }B(z_*,R)\setminus B(z_*,r)\\
F_a&\text{outside }B(z_*,R)
\end{array}\right.,
\end{equation*}
as desired.
\end{proof}

We will need a statement as in Lemma \ref{l:ext_multi_twist} for the map $f_a:L_{p,\q}\to\mathbb{S}^{2n+1}$, rather than for $F_a:\mathbb{S}^{2n+1}\to \mathbb{S}^{2n+1}$.

\begin{cor}\label{c:ext_multi_twist}
Let $f_a:L_{p,\q}\to \mathbb{S}^{2n+1}$ be the multi-twist map defined in \eqref{eq:multi_twist}. Furthermore, fix:
\begin{enumerate}
\item A point $x_*=\pi(r_{1*}e^{\ii \vartheta_{1*}},\ldots,r_{n+1*}e^{\ii \vartheta_{n+1*}})\in L_{p,\q}\setminus B_{f_a}$,

\item an open set $U\subset L_{p,\q}$  containing $x_*$,
\item a ball $B=B(x_*,R)\subset U$.
\end{enumerate}
Assume furthermore that $\overline{B}(f_a(x_*),R)\subset f_a(U)\subset \mathbb{S}^{2n+1}$ and $R<R_0$.

Then there exists a quasiregular RBL map $g_1: U \to \mathbb{S}^{2n+1}$  and a ball  $B'\subset B$ such that:
\begin{enumerate}
\item $g_1(x_*)=f_a(x_*)$,
\item $g_1|_{B'}$ is an isometry (conformal)
\item $g_1|_{U\setminus B}=f_a|_{U\setminus B}$.
\end{enumerate}
Here, $R_0$ and the bound for the radius of $B'$ depend only on $r_{1*}, \cdots, r_{n+1*}$.
\end{cor}

\begin{proof}
Let $V$ be a maximal open ball around $z_*=(r_{1*}e^{\ii \vartheta_{1*}},\ldots,r_{n+1*}e^{\ii \vartheta_{n+1*}})$ such that $\pi(V)\subset U$ and $\pi\vert_V$ is an isometry onto its image.
 We now apply Lemma \ref{l:ext_multi_twist} to $V$, $z_*$, and $F_a$. This provides us with a bound $R_0$. Making $R_0$ smaller if necessary, we may assume that $B(z_*,R)\subset V$ for all $R<R_0$ and thus $B\subset \pi(V)$ for such radii. Then the map $g_1:= G_1 \circ (\pi|_V)^{-1}$ is the desired modification in a neighborhood of $B$, and we may extend $g_1$ by $f_a$ in $U\setminus \pi(V)$.
\end{proof}

\begin{lemma}\label{l:inv_sphere}
Let $\mathbb{S}^{2n+1}$ be the standard sub-Riemannian sphere. Let $U$ be an open neighborhood of a given point $z_0$ in $\mathbb{S}^{2n+1}$. Then there exists an open set $B\subset \overline{B}\subset U$ and a conformal and RBL `inversion' map $\iota:\mathbb{S}^{2n+1}\to \mathbb{S}^{2n+1}$ satisfying:
\begin{enumerate}
\item $\iota^2 = \mathrm{id}$
\item $\iota(B)=\mathbb{S}^{2n+1}\setminus \overline{B}$.
\end{enumerate}
\end{lemma}

\begin{proof}
The map $\iota$ can be obtained by composing the antipodal map
\(
&\iota_0:\mathbb{S}^{2n+1} \to \mathbb{S}^{2n+1}\\
&\iota_0(z_1,\ldots,z_{n+1})=(-z_1,\ldots,-z_{n+1})
\)
with appropriate conformal maps of the sphere (see Examples \ref{ex:rot} and \ref{ex:loxodromic}).
\end{proof}

\subsection{Proof of Theorem \ref{thm:intro:QRexists}}
We are now ready to prove the first of our two main theorems. For an overview of the method, see Example \ref{ex:conformalTrap}. A diagram of mappings involved in the construction is provided in Figure \ref{fig:QRexists}.
\begin{proof}[Proof of Theorem \ref{thm:intro:QRexists}]
Recall that  Theorem \ref{thm:intro:QRexists} claims that every lens space $L_{p,\q}$ admits a nontrivial UQR mapping $g$ with nonempty branch set. Fix $L_{p,\q}$ and a multiple $a$ of $p$.

The map $g$ will be obtained by first modifying the multi-twist map $f:=f_a:L_{p,\mathbf{q}}\to \mathbb{S}^{2n+1}$ defined in \eqref{eq:multi_twist_lens} by applying the flow method on coordinate patches (which we may do since the projection $\pi: \Sp^{2n+1}\rightarrow L_{p,\q}$ is isometric) and then composing the resulting map
$g_1:L_{p,\mathbf{q}}\to \mathbb{S}^{2n+1}$ with an appropriate inversion $\iota:\mathbb{S}^{2n+1}\to \mathbb{S}^{2n+1}$  and the projection $\pi$, so that $g:= \pi \circ \iota \circ g_1$. All appearing mappings are radially bilipschitz (RBL, see \S \ref{sec:QSP}) and hence the same holds true for the composition.

We will, at the same time, obtain a UQR map $G:\mathbb{S}^{2n+1}\to \mathbb{S}^{2n+1}$, defined as $G:= \iota \circ g_1 \circ \pi$, and satisfying the condition
\begin{equation*}
\pi \circ G = g \circ \pi.
\end{equation*}

The idea for this construction stems from \cite{MR1454921,IM2}. In contrast to the situation considered there, the map $f$ which we have to modify has different source and target spaces. The reason why we cannot work directly with the self-map $\pi \circ f$ of $L_{p,\q}$ is that we need to invert, and conformal inversions are available on $\mathbb{S}^{2n+1}$ but not on $L_{p,\q}$. An analogous variant of the original construction appears in \cite{MR1718708}.

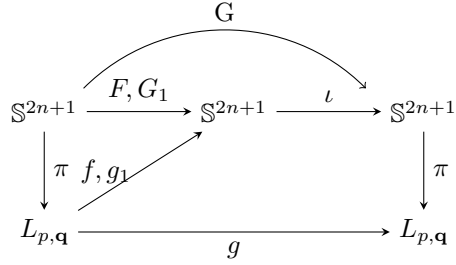
\begin{figure}
\begin{tikzpicture}
  \matrix (m) [matrix of math nodes,row sep=3em,column sep=4em,minimum width=2em]
  {
   \Sp^{2n+1}  & \Sp^{2n+1} &\Sp^{2n+1} \\
   L_{p,\q} & &L_{p,\q}\\};
  \path[-stealth]
  (m-1-1)
	     edge node [above] {$F, G_1$} (m-1-2)
	     edge node [right] {$\pi$} (m-2-1)
   (m-2-1)
            edge node [left] {$f, g_1$} (m-1-2)
	      edge node [below] {$g$} (m-2-3)
  (m-1-2)
 	      edge node [above	] {$\iota$} (m-1-3)
  (m-1-3)
	     edge node [right] {$\pi$} (m-2-3);

  \node[anchor=east] at (2,1) (text) {};
  \node[anchor=east] at (-2,1) (description) {};
  \draw[->] (description) to [out = 45, in = 135, looseness = 1] node[above]{G} (text) ;
\end{tikzpicture}
\caption{Mappings involved in the proof of Theorem \ref{thm:intro:QRexists}.}
\label{fig:QRexists}
\end{figure}

We start by picking a point $x_0 \in L_{p,\q}$ which lies outside $B_f$, is not fixed by $\pi \circ f$ and for which $(\pi \circ f)^{-1}\{x_0\}$ does not meet $B_{\pi \circ f}$. This is generally possible for non-constant quasiregular mappings since such a map is discrete and its branch set is closed and of zero measure. Here we can also make an explicit choice, for instance
\begin{equation*}
x_0 := \pi\left(\tfrac{1}{\sqrt{n+1}}e^{\frac{\ii\pi}{a}},\ldots,\tfrac{1}{\sqrt{n+1}}e^{\frac{\ii\pi}{a}}\right)
\end{equation*}
will work. Let $z_0 \in \mathbb{S}^{2n+1}$ be a point with $\pi(z_0)=x_0$. The inversion $\iota$ will take place at the boundary of a small neighborhood of $z_0$.

The assumption that $x_0$ is not fixed by $\pi \circ f$ ensures that $f(x_0)\neq z_0$. Since quasiregular maps  are discrete and the manifolds under consideration are compact, the preimage of $z_0$ under the map $f$ consists of finitely many points
\begin{equation*}
f^{-1}\{z_0\}=\{x_1,\ldots,x_N\}.
\end{equation*}
For the above choice of $x_0$, this can again be seen directly and one finds that $N=a^{n+1}/p$. As $\pi \circ f$ does not fix $x_0$, the point $x_0$ is distinct from $x_i$ for $i \in \{1,\ldots,N\}$.

Let $U$ and $V$ be disjoint small open neighborhoods about $z_0$ and $f(x_0)$ in $\mathbb{S}^{2n+1}$, respectively, so that
\begin{enumerate}
\item $\pi|_U:U\to \pi(U)$ is injective,
 \item $f^{-1}(U)= \dot{\bigcup}_{i=1}^N U_i$ with $f|_{U_i}:U_i\to U$ homeomorphic for $i\in \{1,\ldots,n\}$,
 \item $f^{-1}(V)$ has a component $V_0$ containing $x_0$ and  $f|_{V_0}:V_0 \to V$ is homeomorphic.
\end{enumerate}
This is possible since $x_0 \notin B_f$, $x_1,\ldots,x_N \notin B_{\pi \circ f}$ and since $f$ is a discrete map between compact manifolds. Notice that $U\cap V = \emptyset$ ensures that $V_0$ does not intersect any of the sets $U_i$.

Choose $R>0$ so that:
\begin{enumerate}
\item For each $i=0, \ldots, N$, $R<R_0(f,x_i)$, where $R_0(f,x_i)$ is the constant provided by Corollary \ref{c:ext_multi_twist} for the map $f$ and point $x_i$,
\item $\overline B_0:=\overline B(x_0, R) \subset V_0$,
\item $\overline B_i:=\overline B(x_i, R) \subset U_i$, for $1\leq i \leq N$,
\item $\overline B(z_0, R) \subset U$,
\item $\overline B(x_0, R) \subset V$.
\end{enumerate}

Corollary \ref{c:ext_multi_twist} then ensures that there exist balls $B_0'\subset B_0$ and $B_i'\subset B_i$, $i\in \{1,\ldots,N\}$ so that we
may modify $f$ in such a way that the new map $g_1$ is quasiregular and RBL
and is given by
\begin{equation*}
g_1:=\left\{\begin{array}{ll}\text{an isometry (conformal)}&\text{on each }B_i'\\\text{quasiconformal interpolation}&\text{on }V_0 \setminus B_0'\text{ and on each }U_i\setminus B_i'\\f&\text{on }L_{p,\q}\setminus \left(V_0 \cup \bigcup_{i=1}^N U_i\right).\end{array}\right.
\end{equation*}
We furthermore have that for every $i=1,\ldots,N$, $g_1(B_i')$ is a ball in $\mathbb{S}^{2n+1}$ of some radius $r< R$, centred at $z_0$ and contained in $U$. Likewise $g_1(B_0')$ is a ball of radius $r$  centred at $f(x_0)$ and contained in $V$.

By Lemma \ref{l:inv_sphere}, there exists a domain $B\subset g_1(B_i')$ and a $1$-quasiconformal inversion $\iota:\mathbb{S}^{2n+1}\to \mathbb{S}^{2n+1}$ which exchanges $B$ with $\mathbb{S}^{2n+1}\setminus \overline{B}$. The composition $\iota \circ g_1:L_{p,\q}\to L_{p,\q}$ is then quasiregular as composition of RBL maps with zero measure branch set, and, since $\pi:\mathbb{S}^{2n+1}\to L_{p,\q}$ is quasiregular and RBL, we find that both
\begin{equation*}
g:= \pi \circ \iota \circ g_1:L_{p,\q} \to L_{p,\q}\quad \text{and}\quad G:= \iota \circ g_1 \circ \pi:\mathbb{S}^{2n+1}\to \mathbb{S}^{2n+1}
\end{equation*}
are quasiregular mappings. It remains to verify that they are in fact \textbf{uniformly} quasiregular.

Recall that $g$ and $G$ are related by the identity $\pi \circ G= g \circ \pi$. This implies $\pi \circ G^n = g^n \circ \pi$ for all $n \in \mathbb{N}$ and thus, since $\pi$ has distortion equal to $1$, $K_{G^n}= K_{g^n}$ (using the chain rule and the analytic criterion in Proposition \ref{p:smooth_K_qr}). It is therefore sufficient to verify that $g$ is uniformly quasiregular.

We now compute the distortion $K_{g^n}$ at different points $x$ by analyzing the orbit of $x$ (note that the set $\pi(B)\subset B_0'$ serves as a ``conformal trap''):
\begin{enumerate}
\item
\label{case:trap}
If $x\in \pi(B)$, then $g_1$ is given by an isometry on a neighborhood of $x$. Since $\iota $ and $\pi$ are conformal, we have that $g$ is conformal at $x$.
Moreover, $g_1 x \notin \overline{B}$, thus $\iota g_1 x \in B$ and $g = \pi \iota g_1 x \in \pi(B)$. Hence, points inside $\pi(B)$ are mapped under iteration again into $\pi(B)$ without picking up any distortion.
\item
\label{case:julia}
If $x\in B_i'$ for some $i\in \{1,\ldots,N\}$, then $g_1$ is conformal at $x$ and thus $x$ does not pick up distortion when mapped by $g= \pi \circ \iota \circ g_1$. The image may lie inside or outside the trap $\pi(B)$.
\item
\label{case:distortion}
Finally, if $x \in L_{p,\q}\setminus \left (\pi(B) \cup \bigcup_{i=1}^N B_i'\right)$, then $g_1 x \notin \overline{B}$ and thus $g x \in \pi(B)$. It is possible that some distortion is picked up in this case, but we do not care, since the point has landed in the conformal trap and hence, as in the first case discussed above, no more distortion can be picked up under further iterations.
\end{enumerate}
We thus have that any point $x$ will see distortion at most once in \eqref{case:distortion}, and then transition to either \eqref{case:julia} or \eqref{case:trap} in which no distortion is picked up. Note that a point may repeat \eqref{case:julia} infinitely many times, without transitioning to \eqref{case:trap}; the Julia set consists exactly of points with such behavior. In either case, we have that distortion is seen by a point at most once, so that $K_{g^n} \leq K_{g}$.

We conclude that the maps $g$ and $G$ are UQR, as desired. Lastly, by construction the branch set of $g$ is nonempty and agrees with that of $f_a$. That is (for $a\neq 1$) it is the image under $\pi$ of the set $\{ (z_1, \ldots, z_{n+1})\in \Sp^{n+1} \st z_i=0 \text{ for some $i$}\}$.

The Julia set of the mapping $g$ constructed above is the Cantor set
\begin{equation}\label{eq:julia}
J(g)=\bigcap_{m=1}^{\infty}\bigcup_{n=m}^{\infty}g^{-n}\left(L_{p,\q}\backslash
\pi(B)\right)\subset \cup_{i=1}^N B_i'.
\end{equation}

For the corresponding statement for a mapping acting in $\bar \R^n$, see \cite{IM2}.
The proof of (\ref{eq:julia}) is completely analogous.

\end{proof}

\section{Differentials on sub-Riemannian manifolds}
\label{sec:differential}
In this section, we recall the Pansu and Margulis--Mostow differentials for quasiconformal mappings between sub-Riemannian manifolds. We then prove that the highest and lowest proper values $\lambda_+, \lambda_-$ of this differential satisfy the standard condition $\lambda_+/\lambda_- \leq K$ for a $K$-quasiconformal mapping.

Most of the complexity involved in defining the Gromov--Hausdorff tangent space for sub-Riemannian manifolds, and the Margulis--Mostow derivative on these spaces, is already visible in simpler contexts. We therefore start with a review of classical differentiation from a metric space perspective, in order to establish notation and emphasize certain conceptual difficulties that motivate the later definitions.

\subsection{Derivatives in $\R^n$}
Recall that the derivative is meant to capture the infinitesimal behavior of a mapping. That is, for a mapping $f: X \rightarrow Y$ of Euclidean spaces, one writes
\[
df \vert_p  = \lim_{h\rightarrow \infty} \delta^Y_{h}\circ f \circ \delta^X_{h^{-1}},
\]
provided that this limit converges uniformly on compacts, and is an affine mapping that takes $p$ to $f(p)$. Here, the mappings $\delta^X_h$ and $\delta^Y_h$ are similarities of $X$ and $Y$ fixing $p$ and $f(p)$ and distorting distances by factor $h$.

For Euclidean spaces, it is, of course, standard to choose in both cases the similarity $\delta_h(\vec x) = h \vec x$ (after normalizing so that $p=f(p)=0$). However, the choice is somewhat arbitrary, as one could zoom in to $p$ and $f(p)$ using other similarities.

\begin{example}
\label{ex:nodiff}
Suppose $X$ and $Y$ are both the complex plane, and take, for $h>0$,
\(
&\delta^X_h(z) \coloneqq hze^{\log h\ii}, &\delta^Y_h(z)\coloneqq hz.
\)
With this choice, the identity map $f(z)=z$ is not differentiable at $0$.
\end{example}

\subsection{Riemannian derivatives via embeddings}
\label{sec:Riemanniandf}
Consider now differentiation of a map $f: X \rightarrow Y$ between Riemannian manifolds at a point $p\in X$.

Except in special cases, $X$ and $Y$ do not admit a family of similarities fixing $p$ and $f(p)$, which makes defining a derivative problematic. One can, abstractly, dilate $X$ and $Y$ by considering rescaled metrics on $X$ and $Y$, namely by taking $(X, hd_X)$ and $(Y, hd_Y)$. However, it is not obvious what meaning to assign to the expression
\[\label{eq:blowup}
\lim_{h\rightarrow \infty} f: (X, hd_X) \rightarrow (Y, hd_Y).
\]

One solution is to assume that both $X$ and $Y$ are embedded in $\R^N$ for some $N$, with their Riemannian metrics induced from the Euclidean inner product. This is always possible, and we may further assume $p=f(p)=0$. We may now think of $df\vert_p$ as a mapping between the tangent spaces $T_pX$ and $T_{fp}Y$ that agrees with $f$ to first order. Before we discuss what this means, we need to characterize $T_pX$ and $T_{fp}Y$.

Consider the orthogonal projection $\pi: \R^N \rightarrow T_pY$. The restriction of $\pi$ to $X$ is Lipschitz, but collapses most of the structure of $X$. However, near $p$, the manifold $X$ is relatively flat, so that $\pi\vert_X$ is locally a homeomorphism onto its image. That is, the restriction of $\pi$ to small balls of $X$ is $L$-bilipschitz, with $L$ tending to $1$ as smaller balls are chosen.

It is useful to rephrase this observation. Namely, consider the manifold $\delta_h X$, naturally isometric to $(X, h d_X)$, and let $\pi_h \coloneqq \pi \vert_{\delta_h X}$. Then, for each compact set $K \subset T_pX$, we have that $\pi^{-1}_h: K \rightarrow \delta_h X \approx (X, h d_X)$ is, for sufficiently large $h$, $L$-bilipschitz, with $L$ tending to $1$ as $h\rightarrow \infty$.

Likewise, we have a sequence of dilated-projection maps $\pi_h: (Y, h d_Y) \rightarrow T_{fp}Y$.

We can now define (in agreement with the standard definition) the map $f: X \rightarrow Y$ to be differentiable at $p$ if
\[
df\vert_p \coloneqq \lim_{h \rightarrow \infty} \pi_h \circ f \circ \pi_h^{-1}
\]
converges uniformly on compacts, and is futhermore a linear mapping. Note that this definition makes precise the expression \eqref{eq:blowup}.

\subsection{Gromov--Hausdorff tangent space and derivative}
We now axiomatize the discussion in \S \ref{sec:Riemanniandf}, reformulating ideas of Gromov  \cite{MR623534}. While our definitions are different from Gromov's, they are equivalent and more convenient for our purposes.

Recall that a map $f:X \rightarrow Y$ between metric spaces is an $(L,C)$-quasi-isometry if one has for some $L\geq 1$, $C\geq 0$ and all points $x, x'\in X$
\(
-C+L^{-1}\norm{x-x'} \leq \norm{fx - fx'} \leq L\norm{x-x'}+C
\)
and furthermore the $C$-neighborhood of $f(X)$ in $Y$ is all of $Y$.

\begin{defi}
\label{defi:TX}
Let $(X, d_X)$ be a metric space and $x\in X$. We say that a metric space $\T_xX$ is the \emph{Gromov--Hausdorff tangent space to $X$ at $x$} if there exists a family of mappings $\pi_h: X \rightarrow \T_xX$ (for $h>0$) such that for every compact $K\subset \T_xX$, the inverse mapping $\pi_h^{-1}: K \rightarrow (X, h d_X)$ exists for sufficiently small $h$ and is an $(L, C)$-quasi-isometry onto its image, with $(L, C) \rightarrow (1,0)$ as $h\rightarrow \infty$. We say that $\{\pi_h\}$ is \emph{asymptotically isometric}.
\end{defi}

\begin{remark}Several remarks are in order:
\begin{enumerate}
\item Gromov--Hausdorff tangent spaces exist only in special classes of spaces. The Gromov--Hausdorff tangent space of a Riemannian manifold (at any point) is a Euclidean space of the same dimension, as is evident from \S \ref{sec:Riemanniandf}. The isometry class of $\T_xX$ may vary with $x$.
\item We are only interested in asymptotic properties of $\pi_h$, so in principle the mapping only needs to be defined on some small neighborhood of $x$.
\item One should keep in mind that the sequence $\pi_h$ is asymptotically isometric as a family of mappings to \emph{rescaled} versions of $X$.
\end{enumerate}
\end{remark}

\begin{defi}
Suppose $\T_xX$ is the Gromov--Hausdorff tangent space at $x$ of a metric space $X$. Fix an asymptotically isometric sequence $\pi_h$. We say that a family of points $x_h \in X$ \emph{converges to $v\in \T_xX$} if and only if $\lim_{h \rightarrow \infty} \pi_h(x_h)=v$. Note in particular that this implies that $x_h \rightarrow x$ in $(X,d_X)$.
\end{defi}
\begin{remark}
\label{rmk:curves}
One can use the definition of convergence to identify each point of $\T_xX$ with an equivalence class of ``curves'' in $X$, and write the metric on $\T_xX$ purely in terms of the metric of $X$.
\end{remark}

\begin{defi}
\label{defi:df}
Let $f: X\rightarrow Y$, and fix asymptotically isometric sequences $\pi_h$ at $x$ and $fx$. Then $f$ is \emph{differentiable at $x$} if the limit
\[
df\vert_x \coloneqq \lim_{h \rightarrow \infty} \pi_h \circ f \circ \pi_h^{-1}
\]
converges uniformly on compacts.
\end{defi}

\begin{remark}
In certain settings, one further requires that the derivative $df$ satisfies additional conditions, e.g.\ linearity.
\end{remark}

The following results are immediate:

\begin{lemma}[Curve interpretation of the derivative]
\label{lemma:curve}
Suppose $f$ is differentiable at $x$ and we have a family of points $\{x_h\}\in X$ with $x_h \rightarrow v \in \T_xX$. Then $f(x_h) \rightarrow df(v)$.
\end{lemma}

\begin{lemma}[Chain rule]
\label{lemma:chain}
Suppose $f: X \rightarrow Y$ is differentiable at $x\in X$, and $g: Y \rightarrow Z$ is differentiable at $f(x)\in Y$. Then $g \circ f$ is differentiable at $x$, with derivative $dg\vert_{fx} \circ df\vert_x$.
\end{lemma}

\subsection{Riemannian tangent spaces via coordinate charts}
\label{sec:Riemannian-chart}
We will be interested in derivatives of mappings between sub-Riemannian manifolds. In that context, the Nash embedding theorem is not available, so the Gromov--Hausdorff tangent space is found by working in charts. In preparation for the sub-Riemannian construction, we first construct Riemannian tangent spaces in a new way.

Let $X$ be a Riemannian manifold of dimension $n$, and $x\in X$. Since $X$ is a smooth manifold, there is a diffeomorphism from a small neighborhood of $x$ to an open set $U$ in $\R^n$, which we denote by $\pi_1$. Composing with an affine transformation if necessary, we may assume $\pi_1(x)=0$, and that the Riemannian inner product at $x$ agrees with the standard inner product on $\R^n$.

By the smoothness of the Riemannian metric, we have that for sufficiently small balls $K$ around $0$, the map $\pi_1^{-1}\vert_K$ is an $(L,1)$-quasi-isometry onto its image, with $L \rightarrow 1$ as the diameter of $K$ vanishes.
It is then clear that $\pi_h \coloneqq \delta_h \circ \pi_1$ is an asymptotically isometric family of mappings from $(X, hd_X)$ to $\R^n$, so that we have $\R^n = \T_xX$.

\begin{remark}To agree with previous notation, we may extend $\pi_1$ to a map on all of $X$ by sending unassigned points to an arbitrary nonzero point in $\R^n$.
\end{remark}

\subsection{Sub-Riemannian derivatives}
We are now ready to discuss derivatives in the sub-Riemannian setting. We start with:
\begin{thm}[Mitchell \cite{Mitchell}]
\label{thm:Mitchell}
Let $M$ be an equiregular sub-Riemannian manifold, and $p\in M$. Then the Gromov--Hausdorff tangent space $\T_pM$ exists and is isometric to a Carnot group.
\end{thm}

We now briefly sketch the proof of Theorem \ref{thm:Mitchell}, in part to establish notation. Note that the proof follows the ideas presented in \S \ref{sec:Riemannian-chart}. Consider a Riemannian metric on $M$ whose restriction to $HM$ provides the given sub-Riemannian inner product on $HM$. As in \S \ref{sec:Riemannian-chart}, let $\pi_1: M \rightarrow \R^n$ be a coordinate patch for $M$ that maps $p$ to $0$ and so that the Riemannian metric on $M$ agrees, at the origin, with the standard metric on $\R^n$. We may further assume that \emph{at the origin}, the horizontal bundle $HM$ is spanned by the vectors $\partl{x_1}, \ldots, \partl{x_{q_1}}$; that the bundle $\langle HM, [HM,HM]\rangle$ is spanned by $\partl{x_1}, \ldots, \partl{x_{q_1+q_2}}$, and so on. This allows us to partition the coordinate $( x_1, \ldots,  x_n)$ as $(\vec x_1, \ldots, \vec x_r)$, with vector $\vec x_i$ having length $q_i$. Such coordinates are call \emph{privileged coordinates}. We can define a dilation on $\R^n$ by
\[
\delta_h ( \vec x_1, \ldots, \vec x_r) \coloneqq (h \vec x_1, \ldots, h^r \vec x_r).
\]
Mitchell shows that as one expands the coordinate chart using $\delta_h$, the bundle $HM$ and its metric limit to a new sub-Riemannian metric that is isometric to a Carnot group $G$ (for which $\delta_h$ is a similarity). Furthermore, Mitchell shows that the map $\pi_h \coloneqq \delta_r \circ \pi_1$ is an asymptotic isometry, in the sense of Definition \ref{defi:TX}, so that $G$ is the Gromov--Hausdorff tangent space to $M$ at $p$.

\begin{remark}
\label{remark:HM}
Under the dilations $\delta_h$, the horizontal subspace at $p$,  $H_pM$, is preserved. Furthermore, the metric on $\T_pM$ is the left-invariant metric that agrees with the metric on $H_pM$ at $p$. It follows that the distance $d_{\T_pM}(0, (\vec x_1, 0, \ldots, 0))$ is given by the sub-Riemannian norm $\norm{\vec x_1}_M$, where $x_1$ is thought of as a vector in $H_pM$. It is therefore natural to identify the subspace $\R^{q_1} \subset \T_pM$ with the horizontal subspace $H_pM$.
\end{remark}

Margulis--Mostow used the mappings $\pi_h$ to define differentiation on sub-Rie{\-}mann{\-}ian manifolds (see also Definition \ref{defi:df}):
\begin{defi}[Margulis--Mostow \cite{MargulisMostow95}]
\label{defi:MMdf}
The \emph{differential} $df\vert_p$ of a mapping $f$ between equiregular sub-Riemannian manifolds is the limit
\[
df\vert_p \coloneqq \lim_{h\rightarrow\infty} \pi_h \circ f \circ \pi_h^{-1},
\]
provided it converges uniformly on compacts, is a Lie group homomorphism of the Carnot groups that serve as the Gromov--Hausdorff tangent spaces, and is furthermore equivariant with respect to their homotheties. Here, the mappings $\pi_h$ are defined as above via privileged coordinates at $p$.
\end{defi}

Using the curve interpretation of the tangent space (see \ref{rmk:curves}), Margulis--Mostow showed that the definition is, in fact, independent of the particular choice of privileged coordinates used to define differentiability. They then showed:

\begin{thm}[Margulis--Mostow \cite{MargulisMostow95}]
\label{thm:MM}
Let $f$ be a quasiconformal mapping between equiregular sub-Riemannian manifolds. Then the Margulis--Mostow derivative $df\vert_p$ exists at almost every point $p\in M$. Furthermore,
$df\vert_p$ is an isomorphism  for almost every $p$.
\end{thm}

We will strengthen Theorem \ref{thm:MM} somewhat in the remainder of this section and in \S \ref{sec:invariant}. We start with a restriction on the ratio of ``proper values'' of $df\vert_p$:

\begin{lemma}
\label{lemma:df-qc}
Let $f: M \rightarrow N$ be a $C$-quasiconformal mapping between equiregular sub-Riemannian manifolds. If $f$ is differentiable at a point $p$ with injective derivative, then the derivative $df\vert_p: \T_pM \rightarrow \T_{fp}N$ satisfies for all points $u,v \in \T_pM$ at distance $1$ from $0$ (with respect to $d_{\T_pM}$):
\(
\frac{d_{\T_pN}(0, df\vert_p u)}{d_{\T_pN}(0, df\vert_p v)} \leq C \text{ }.
\)
\begin{proof}
Fix a sequence of asymptotically isometric mappings $\pi_h$, as in Definition \ref{defi:MMdf}, and let, for each $h$,
\(
&u_h \coloneqq \pi_h^{-1}u &v_h \coloneqq \pi_h^{-1}v
\)
We have, by definition of asymptotically isometric families, for some $(L, K)$ depending on $h$:
\(
-K + L^{-1} d_{\T_pM}(0,u) \leq h d_M(p, u_h) \leq Lh d_{\T_pM}(0,u) + K
\)
We also have a corresponding inequality for each $v_h$.

We would like to further assume that $d_M(p, u_h) = d_M(p, v_h)$ for each $h$. This can be enforced using a small perturbation along geodesics joing $u_h$ and $v_h$ to $p$. It is easy to see that the resulting points, which we continue to denote $u_h$ and $v_h$ converge to $u$ and $v$, respectively.

We conclude from the definition of differentiability (see Lemma \ref{lemma:curve}) that $f(u_h)$ and $f(v_h)$ converge, respectively, to $df(u)$ and $df(v)$.
\(
\frac{d_{\T_pN}(0, df\vert_p u)}{d_{\T_pN}(0, df\vert_p v)} \leq \frac{L h d_N(fp, fu_h) +K }{L^{-1} h d_N(fp, fv_h)-K}
\)

Furthermore, we have that $d_M(p, u_h) = d_M(p, v_h)$ goes to $0$ as $h \rightarrow \infty$. In particular, by the definition of quasiconformality, for each $\epsilon$ and sufficiently large $h$ we have:
\(
\frac{d_N(fp, fu_h)}{d_N(fp, fv_h)} \leq C+\epsilon.
\)

We continue the above estimate:
\(
\frac{d_{\T_pN}(0, df\vert_p u)}{d_{\T_pN}(0, df\vert_p v)} &\leq \frac{L h d_N(fp, fu_h) +K }{L^{-1} h d_N(fp, fv_h)-K} \\
&\leq \frac{L h (C+\epsilon)d_N(fp, fv_h)+K}{L^{-1} h d_N(fp, fv_h)-K}
\)
As $h\rightarrow \infty$, we have that $L \rightarrow 1$, $K\rightarrow 0$, and $h d_N(fp, fv_h)\rightarrow d_{\T_{fp}N}(0,df(v))\neq 0$. Thus, the quotient is bounded by $C$, as desired.
\end{proof}
\end{lemma}

We can now state a differentiability theorem for quasiregular mappings (we will show $df$ is measurable in \S \ref{sec:invariant}):

\begin{thm}
\label{thm:QRdf}
Let $f: M \rightarrow N$ be a $C$-quasiregular mapping between equiregular sub-Riemannian manifolds. Then the Margulis--Mostow derivative $df\vert_p$ exists at almost every point $p$, and furthermore one has, for almost every $p$ and any $u, v\in \T_pM$ with $d_{\T_pM}(0, u) = d_{\T_pM}(0, v) =1$ that
\(
\frac{d_{\T_pN}(0, df\vert_p u)}{d_{\T_pN}(0, df\vert_p v)} \leq C.
\)
\begin{proof}
A quasiregular mapping is a local homeomorphism away from the branch set. Thus, each point away from the branch set has a neighborhood on which $f$ is a quasiconformal mapping onto its image, cf.\ Proposition \ref{p:qr_loc_qc}. Differentiability follows by Theorem \ref{thm:MM}. The distortion estimate follows from Lemma \ref{lemma:df-qc}.
\end{proof}
\end{thm}

By Remark \ref{remark:HM} we can identify the first layer of $\T_pM$ with the horizontal space $H_pM$. We then have:

\begin{cor}
\label{cor:QRdf}
Let $f: M \rightarrow N$ be a $C$-quasiregular mapping between equiregular sub-Riemannian manifolds. Then at almost every $p$, the mapping $f$ induces an isomorphism of vector spaces $df\vert_p: H_pM \rightarrow H_pN$ satisfying (with respect to the given inner products on $HM$ and $HN$):
\(
\frac{\norm{df\vert_p u}/\norm{u}}{\norm{df\vert_p v}/\norm{v}} \leq C \text{ for all $u, v\in H_pN$}.
\)
\end{cor}

\begin{remark}
Lemma \ref{lemma:chain} provides a chain rule that applies to all the derivatives in this section (Theorem \ref{thm:MM}, Theorem \ref{thm:QRdf}, Corollary \ref{cor:QRdf}).
\end{remark}

\section{Invariant conformal structure}\label{sec:invariant}
The goal of this section is to prove Theorem \ref{thm:intro-conformal}, which states that every UQR mapping between equiregular sub-Riemannian manifolds admits an invariant conformal structure. A similar result was first suggested by Sullivan \cite{aSu} and carried out by Tukia \cite{MR861709} for quasiconformal groups, and later extended to certain quasiregular semigroups by Iwaniec and Martin \cite{MR1404085}. While Tukia explicitly considers only quasiconformal groups on $\Sp^n$, his proof is more general.

\subsection{Tukia's theorem}
Let $M$ be a manifold and $B$ a finite-dimensional vector bundle over $M$, with a fixed choice of inner product (on the fibers). For $p\in M$, denote the fiber over $M$ by $F_p$. Recall that a map $f: B \rightarrow B$ is a \emph{bundle map} if $f$ sends each fiber $F_p$ to another fiber (which we denote $F_{fp}$), and the restriction $f_p:= f\vert_{F_p}: F_p \rightarrow F_{fp}$ is linear. Since $f$ takes fibers to fibers, it induces a map $f\vert_M: M \rightarrow M$, which we continue to denote as $f$.

We say that a vector bundle map $f: B \rightarrow B$ has \emph{bounded distortion} if
\[\label{def:bounded-distortion} \sup_{p\in M} \frac{ \sup_{v\in F_p} \norm{fv}/\norm{v}} {\inf_{v\in F_p} \norm{fv}/\norm{v}}=C < \infty.\]
Furthermore, $f$ has \emph{uniformly bounded distortion} if \eqref{def:bounded-distortion} holds for all iterates $f^{n}$ of $f$, with $C$ independent of $n$.

The bundle $B$ \emph{admits an $f$-invariant measurable conformal structure} if there exists an inner product $\langle \cdot, \cdot\rangle$ on the fibers of $B$, varying measurably, such that for some positive measurable function $\lambda$ on $M$ we have
\[
\langle fu, fv \rangle  = \lambda \langle u, v\rangle
\]
for all $u,v \in F_p$ for almost every fiber $F_p$. Typically, one also assumes a boundedness condition for $\langle \cdot,\cdot \rangle$, which can be expressed in terms of the matrix-valued function $s$ associated to $\langle \cdot,\cdot \rangle$.

The following theorem is a rephrasing of Tukia's core result in \cite{MR861709}. We sketch its proof to adapt it to our terminology.
\begin{thm}
\label{thm:tukia}
Let $f: B \rightarrow B$ be a bundle map of uniformly bounded distortion. Then $B$ admits an $f$-invariant measurable conformal structure.
\end{thm}
\begin{proof}[Sketch of proof]
Denote the given inner product on $B$ by $\langle \cdot, \cdot \rangle_0$ and furthermore fix an orthonormal basis $\mathcal B_p$ at each point $p$, varying measurably. Let $d$ be the dimension of the fibers of $B$.

We would like to show that there exists an inner product $\langle \cdot, \cdot \rangle$ on $B$ that is $f$-invariant, up to a multiplicative factor. That is, we would like the property that for all $u,v$ in each fiber $F_p$ and for some positive function $\lambda$ on $M$,
\[\label{eq:invariantinnerproduct}
\langle f_pu, f_pv \rangle  = \lambda(p) \langle u, v\rangle.
\]
Using the basis $\mathcal B_p$, we may find a positive definite matrix $s_p\in GL(d,\R)$ so that
\[
\langle u, v\rangle = \langle u, s_p v\rangle_0
\]
and the invariance relation \eqref{eq:invariantinnerproduct} becomes (taking transposes using $\mathcal B$)
\[
\langle f_pu, s_{fp} f_p  v \rangle_0  &= \lambda(p) \langle u, s_p v\rangle_0,\\
\langle u, f_p^t s_{fp} f_p v \rangle_0  &= \lambda(p) \langle u, s_p v\rangle_0,\\
f_p^t s_{fp} f_p &= \lambda(p) s_p.
\]
We now assume $s_p$ has determinant one, so that \eqref{eq:invariantinnerproduct} further reduces to
\[
\label{eq:invariantmatrices2}
(\det f_p)^{-2/d}  f_p^t s_{fp} f_p = s_p.
\]

Note now that $s_p$ needs to be an element of the space $S \subset SL(d,\R)$ of positive definite symmetric matrices. Tukia points out that $S$ may be identified with $SL(d,\R)/SO(d)$, a non-negatively curved symmetric space. Normalized transpose-conjugation by an element of $GL(d,\R)$, as in \eqref{eq:invariantmatrices2}, is an isometry of $S$.

We are now ready to construct $s_p$. Consider first the orbit of $p$ under f:
\[ \mathcal O(p)= \{ f^np \st n \in \N\}.\]
Note that we have $\mathcal O(fp) = f\mathcal O(p)$.
Now, consider the transformations $(f^n)_p : F_p \rightarrow F_{f^np}$ as elements of $S$:
\[\mathcal S(p) = \{ (\det~(f^n)_p)^{-2/d}  (f^n)_p^t (f^n)_p \st n\in \N\}\subset S.\]
We obtain the invariance equation $\mathcal S(fp) = f_p \cdot \mathcal S(p) = (\det f_p)^{-2/d}f_p^t\mathcal S(p)f_p$.

Under the standard metric on $S$, the uniformly-bounded condition on $f$ gives us that $\mathcal S(p)$ is a bounded set in $S$. Tukia shows that every bounded set in $S$ is contained in a \emph{unique} ball of minimum radius. Take $s_p \in S$ to be the center of the unique ball of minimum radius containing $\mathcal S(p)$. Because transpose-conjugation is an isometry in $S$, the invariance relation on $\mathcal S(p)$ turns into \eqref{eq:invariantmatrices2}.

We thus have that the conformal class of $\langle u, v \rangle := \langle u, s_p v\rangle$ is $f$-invariant. It remains to show that it is also measurable. It is clear that the map $p \mapsto \mathcal S(p)$ is measurable, and Tukia shows that ``averaging'' operation $\mathcal{S}(p)\mapsto s_p$ is continuous with respect to the Hausdorff topology on subsets of $S$.
\end{proof}

\subsection{Bundle maps}
\label{sec:bundles}
We now prove Theorem \ref{thm:intro-conformal} through a series of lemmas focusing on the tangent bundle. Recall that we start with a UQR self-map $f: M \rightarrow M$ of an equiregular sub-Riemannian manifold with horizontal bundle $HM$.

\begin{remark}We note first that Margulis--Mostow provided an intrinsic definition of the Gromov--Hausdorff tangent bundle of $M$ by means of equivalence classes of curves in \cite{MargulisMostow2000}. By viewing $M$ in privileged coordinates from each point, one sees that the Gromov--Hausdorff tangent bundle of $M$ is indeed a bundle whose fibers are all homeomorphic to $\R^d$, where $d$ denotes the dimension of the manifold $M$. One sees that the Gromov tangent bundle is vector-bundle-isomorphic to $TM$ by viewing the fibers in turn as Carnot groups, then as their Lie algebras, and finally simply as vector spaces.
\end{remark}

\begin{lemma}The UQR map $f$ induces a measurable mapping $df$ on the Gromov tangent bundle to $M$.
\label{lemma:UQRdf}
\begin{proof}
Recall that the branch set $B_f$ of $f$ is the set of points where $f$ is not a local homeomorphism. By definition of quasiregularity, $B_f$ is a null set. Away from it, $f$ is a local homeomorphism, so locally quasiconformal. By Theorem \ref{thm:QRdf} we then have that $f$ is almost everywhere differentiable.

We now claim that the derivative $df$ varies \emph{measurably}. Classically (in the Euclidean case), measurability of derivatives is shown to exist by dilating the source and target of $f$. Because our dilations vary from point to point, we have to define some additional structure. Namely, we will observe the dilations at all points simultaneously by building an appropriate bundle over $M$.

Let $d$ be the dimension of $M$, and consider the bundle $C(TM, \R^d)$, whose fibers can be identified with the set of continuous maps from the $\R^d$ to $\R^d$. Consider the section $\sigma$ of $C(TM, \R^d)$ that associates to each point $p$ the map $f_p: \R^d \rightarrow \R^d$ given by privileged coordinates at $p$. Intuitively, $f_p$ is the ``view of $f$ from the point p''. It is clear that this section is measurable, since the choice of privileged coordinates can be made smoothly, at least locally.

The Gromov tangent space at each point $p$ is defined by dilating the privileged coordinates at $p$ by a map $\delta^p_r$. We may view $\delta_r$ as a section of $C(TM, \R^d)$. Composing, we have a measurable section $\delta_r \circ \sigma \circ \delta_{r^{-1}}$. We now take the limit
\[ \sigma_\infty := \limsup_{r\rightarrow \infty} \delta_r \circ \sigma \circ \delta_{r^{-1}}\]
defined as the $\limsup$ along each coordinate. We have thus taken the derivative simultaneously at every point. The section $\sigma_\infty$ is given by a $\limsup$ of measurable functions, so it is measurable. Lastly, we know that the derivative $df$ exists almost everywhere. By construction, $\sigma_\infty$ must agree with the derivative where $df$ exists. In particular, we have that $df$ is measurable.
\end{proof}
\end{lemma}

\begin{lemma}The UQR map $f$ induces a measurable bundle map $df: HM \rightarrow HM$. Furthermore, there exists a full-measure set $U\subset M$ such that the restriction of $df$ to the bundle over $U$ is an  bundle map of uniformly bounded distortion.
\label{lemma:uqrBundle}
\begin{proof}
Lemma \ref{lemma:UQRdf} states that the Margulis--Mostow derivative $df$ is measurable. Its restriction to $HM$ remains measurable.

To obtain a bundle map of uniformly bounded distortion, we need to restrict to a smaller full-measure set. Namely, we remove points where $d(f^n)$ is not defined for any $n$. We further remove points where, for any $n$, $d(f^n)$ does not satisfy the distortion bound in Corollary \ref{cor:QRdf}. Note in particular that we are removing the branch set and its orbit. The chain rule Lemma \ref{lemma:chain} further states that the bounds for $d(f^n)\vert_{HM}$ remain true for $(df)^n\vert_{HM}$.  Lastly, we remove the pre- and forward images of the set just removed. Lemma \ref{l:lusin} ensures that after this process, we retain a set $U$ of full measure.
\end{proof}
\end{lemma}

\begin{proof}[Proof of Theorem \ref{thm:intro-conformal}]
The theorem is now a direct consequence of Lemma \ref{lemma:uqrBundle} and Theorem \ref{thm:tukia}. To summarize, we used the Margulis--Mostow derivative to obtain a measurable differential $df$, defined on the Gromov--Hausdorff tangent spaces. We showed that it restricts to the bundle $HM$, and almost everywhere satisfies appropriate distortion bounds. After restricting to a full-measure set $U$, we applied Theorem \ref{thm:tukia} to obtain an invariant conformal structure over $U$.
\end{proof}

\section*{Acknowledgements}
The authors would like to thank Jeremy Tyson for useful conversations. Much of the work was conducted while the second author visited Aalto University, and he would like to thank the department for its hospitality.

\bibliographystyle{alpha}
\bibliography{qc_1}

\end{document}